\documentclass[reqno]{amsart}

\usepackage[danish,english]{babel}
\usepackage[latin1]{inputenc}
\usepackage[T1]{fontenc}
\usepackage{amsmath,amsthm,amssymb,amsfonts,mathrsfs,latexsym}
\usepackage{graphicx}
\usepackage{verbatim}
\usepackage[all]{xy}
\usepackage[colorlinks=true]{hyperref}
\usepackage{times}
\usepackage{appendix}

\DeclareFontFamily{U}{mathx}{\hyphenchar\font45}
\DeclareFontShape{U}{mathx}{m}{n}{
      <5> <6> <7> <8> <9> <10>
      <10.95> <12> <14.4> <17.28> <20.74> <24.88>
      mathx10
      }{}
\DeclareSymbolFont{mathx}{U}{mathx}{m}{n}
\DeclareFontSubstitution{U}{mathx}{m}{n}
\DeclareMathAccent{\widecheck}{0}{mathx}{"71}
\DeclareMathAccent{\wideparen}{0}{mathx}{"75}

\newcommand{\hooklongrightarrow}{\lhook\joinrel\longrightarrow}

\makeatletter
\newtheorem*{rep@theorem}{\rep@title}
\newcommand{\newreptheorem}[2]{%
\newenvironment{rep#1}[1]{%
 \def\rep@title{#2 \ref{##1}}%
 \begin{rep@theorem}}%
 {\end{rep@theorem}}}
\makeatother

\newtheorem{thm}{Theorem}[section]
\newtheorem{thmA}{Theorem}
\newreptheorem{thm}{Theorem}
\newtheorem{lem}[thm]{Lemma}
\newtheorem{prop}[thm]{Proposition}
\newtheorem{cor}[thm]{Corollary}
\newtheorem*{thm*}{Theorem}
\newtheorem*{claim*}{Claim}

\newtheorem{innercustomthm}{Theorem}
\newenvironment{customthm}[1]
  {\renewcommand\theinnercustomthm{#1}\innercustomthm}
  {\endinnercustomthm}

\theoremstyle{definition}
\newtheorem{defi}[thm]{Definition}
\newtheorem{exam}[thm]{Example}

\newtheorem{rem}[thm]{Remark}

\newcommand{\mc}[1]{\mathcal{#1}}

\newcommand{\mb}[1]{\mathbb{#1}}

\newcommand{\mr}[1]{\mathrm{#1}}

\newcommand{\be}{\begin{align*}}
\newcommand{\ee}{\end{align*}}
\newcommand{\C}{\mathbb{C}}
\newcommand{\N}{\mathbb{N}}

\newcommand{\R}{\mathbb{R}}
\newcommand{\Z}{\mathbb{Z}}
\newcommand{\F}{\mathbb{F}}

\newcommand{\la}{\langle}
\newcommand{\ra}{\rangle}
\renewcommand{\epsilon}{\varepsilon}
\renewcommand{\phi}{\varphi}
\renewcommand{\tilde}{\widetilde}
\renewcommand{\hat}{\widehat}
\renewcommand{\Re}{\mr{Re}}

\newcommand{\Cs}{\ensuremath{C^*}}
\newcommand{\cb}{\mr{cb}}
\newcommand{\WA}{\mr{WA}}
\newcommand{\WH}{\mr{WH}}

\newcommand{\conv}{\mr{conv}}
\newcommand{\tensor}{\otimes}
\newcommand{\vntensor}{\bar\otimes}

\newcommand{\II}{\mr{II}}

\DeclareMathOperator{\SL}{SL}
\DeclareMathOperator{\SO}{SO}
\DeclareMathOperator{\SU}{SU}
\DeclareMathOperator{\Sp}{Sp}

\DeclareMathOperator{\id}{id}
\DeclareMathOperator{\supp}{supp}

\numberwithin{equation}{section}
\begin{document}
\selectlanguage{english} % Select either of the sets loaded with babel

\begin{abstract}
We introduce the weak Haagerup property for locally compact groups and prove several hereditary results for the class of groups with this approximation property. The class contains a priori all weakly amenable groups and groups with the (usual) Haagerup property, but examples are given of groups with the weak Haagerup property which are not weakly amenable and do not have the Haagerup property.

In the second part of the paper we introduce the weak Haagerup property for finite von Neumann algebras, and we prove several hereditary results here as well. Also, a discrete group has the weak Haagerup property if and only if its group von Neumann algebra does.

Finally, we give an example of two $\II_1$ factors with different weak Haagerup constants.
\end{abstract}

%%%%%%%%% Title %%%%%%%%%

\title{The weak Haagerup property}
\author{S{\o}ren Knudby}
\thanks{Supported by ERC Advanced Grant no.~OAFPG 247321 and the Danish National Research Foundation through the Centre for Symmetry and Deformation (DNRF92).}
\address{Department of Mathematical Sciences, University of Copenhagen,
\newline Universitetsparken 5, DK-2100 Copenhagen \O, Denmark}
\email{knudby@math.ku.dk}
\date{\today}
\maketitle
%%%%%%%%%% Text starts here %%%%%%%%%%
%\setcounter{tocdepth}{1}
\tableofcontents

\newpage
\section{Introduction}\label{sec:introduction}
In connection with the famous Banach-Tarski paradox, the notion of an amenable group was introduced by von Neumann \cite{amenable}, and since then the theory of amenable groups has grown into a huge research area in itself (see the book \cite{MR767264}). Today, we know that amenable groups can be characterized in many different ways, one of which is the following. A locally compact group $G$ is amenable if and only if there is a net $(u_\alpha)_{\alpha\in A}$ of continuous compactly supported positive definite functions on $G$ such that $u_\alpha \to 1$ uniformly on compact subsets of $G$ (see \cite[Chap.~2, Sec.~8]{MR767264}). When formulated like this, amenability is viewed as an approximation property, and over the years several other (weaker) approximation properties resembling amenability have been studied. For a combined treatment of the study of such approximation properties we refer to \cite[Chapter~12]{MR2391387}. We mention some approximation properties below and relate them to each other (see Figure~\ref{fig:aps}).

Recall that a locally compact group $G$ is \emph{weakly amenable}, if there is a net $(u_\alpha)$ of compactly supported Herz-Schur multipliers on $G$, uniformly bounded in Herz-Schur norm, such that $u_\alpha\to 1$ uniformly on compacts. The least uniform bound on the norms of such nets (if such a bound exists at all) is the weak amenability constant of $G$. We denote the weak amenability constant (also called the Cowling-Haagerup constant) by $\Lambda_\WA(G)$. The notation $\Lambda_G$ and $\Lambda_{\cb}(G)$ for the weak amenability constant is also found in the literature. For the definition of Herz-Schur multipliers and the Herz-Schur norm we refer to Section~\ref{sec:prelim}, but let us mention here that any (normalized) positive definite function on the group $G$ is a Herz-Schur multiplier (of norm 1). Hence all amenable groups are also weakly amenable (how lucky?) and their weak amenability constant is $1$. If a group is not weakly amenable we write $\Lambda_\WA(G) = \infty$.

If, in the definition of weak amenability, no condition were put on the boundedness of the norms, then any $G$ group would admit such a net of functions approximating $1$ uniformly on compacts: It follows from Lemma~3.2 in \cite{MR0228628} that given any compact subset $K$ of a locally compact group $G$, there is a compactly supported Herz-Schur multiplier $u$ taking the value $1$ on all of $K$. The lemma in fact states something much stronger, namely that one can even arrange for $u$ to be in the linear span of the set of continuous compactly supported positive definite functions. But the Herz-Schur norm of $u$ will in general not stay bounded when the compact set $K$ grows.

Weak amenability of groups has been extensively studied. Papers studying weak amenability include \cite{MR748862}, \cite{MR2132866}, \cite{MR996553}, \cite{MR784292}, \cite{MR1245415}, \cite{MR1418350}, \cite{MR520930}, \cite{UH-preprint}.

The \emph{Haagerup property} is another much studied approximation property (see the book \cite{MR1852148}). It appeared in connection with the study of approximation properties for operator algebras (see e.g. \cite{MR520930} and \cite{MR718798}).
It is known that groups with Haagerup property satisfy the Baum-Connes conjecture \cite{MR1487204}, \cite{MR1821144}.
The definition is as follows.

A locally compact group $G$ has the Haagerup property, if there is a net $(u_\alpha)$ of continuous positive definite functions on $G$ vanishing at infinity such that $u_\alpha\to 1$ uniformly on compacts. It is clear that amenability implies the Haagerup property, but the free groups demonstrate that the converse is not true (see \cite{MR520930}). It is however not clear what the relation between weak amenability and the Haagerup property is. When Cowling and Haagerup proved that the simple Lie groups $\Sp(1,n)$ are weakly amenable \cite{MR996553}, it became clear that weak amenability does not imply the Haagerup property, because these groups also have Property (T) when $n\geq 2$ (see \cite{MR0245725},\cite{MR0399361},\cite{MR2415834}), and Property (T) is a strong negation of the Haagerup property. However, since the weak amenability constant of $\Sp(1,n)$ is $2n -1$, it does not reveal if having $\Lambda_\WA(G) = 1$ implies having the Haagerup property.

In the light of the approximation properties described so far, and in order to study the relation between weak amenability and the Haagerup property, the \emph{weak Haagerup property} was introduced (for discrete groups) in \cite{MR3146826}. The class of groups with the weak Haagerup property encompasses in a natural way all the weakly amenable groups and groups with the Haagerup property. The definition goes as follows (see also Definition~\ref{defi:WH}).

A locally compact group $G$ has the \emph{weak Haagerup property}, if there is a net $(u_\alpha)$ of Herz-Schur multipliers on $G$ vanishing at infinity and uniformly bounded in Herz-Schur norm such that $u_\alpha\to 1$ uniformly on compacts. The least uniform bound on the norms of such nets (if such a bound exists at all) is the \emph{weak Haagerup constant} of $G$, denoted $\Lambda_\WH(G)$.

In the same way that one deduces that amenable groups are weakly amenable, one sees that groups with the Haagerup property also have the weak Haagerup property. Also, it is trivial that $1\leq \Lambda_\WH(G) \leq \Lambda_\WA(G)$ for every locally compact group $G$, and in particular all weakly amenable groups have the weak Haagerup property.

It is not immediately clear if the potentially larger class of groups with the weak Haagerup property actually contains groups which are not weakly amenable and at the same time without the Haagerup property. In Corollary~\ref{cor:class} we will demonstrate that this is the case.

There are many examples of groups $G$ where $\Lambda_\WH(G) = \Lambda_\WH(G)$, e.g. all amenable groups and more generally all groups $G$ with $\Lambda_\WA(G) = 1$. There are also examples where the two constants differ. In fact, the wreath product group $H = \Z/2 \wr \F_2$ of the cyclic group of order two with the non-abelian free group of rank two is such an example. The group $H = \Z/2 \wr \F_2$ is defined as the semidirect product of $\bigoplus_{\F_2} \Z/2$ by $\F_2$ where $\F_2$ acts on $\bigoplus_{\F_2} \Z/2$ by the shift action. It is known that $H$ has the Haagerup property (see \cite{MR2393636}), and hence $\Lambda_\WH(H) = 1$. But in \cite[Corollary 2.12]{MR2680430} it was shown that $\Lambda_\WA(H) \neq 1$. It was later shown in \cite[Corollary 4]{MR2914879} that in fact $\Lambda_\WA(H) = \infty$.

There is another approximation property of locally compact groups that we would like to briefly mention. It is called the \emph{Approximation Property} or simply AP and was introduced in \cite{MR1220905} (see the end of Section~\ref{sec:prelim} for the definition). It is known that all weakly amenable groups have AP, and there are non-weakly amenable groups with the AP as well (see \cite{MR1220905}).

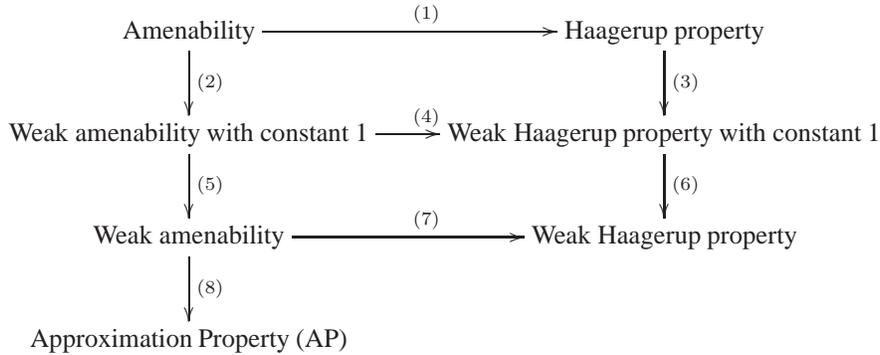
\begin{figure}[th]
$$
\xymatrix{
\text{Amenability} \ar[r]^{(1)} \ar[d]^{(2)}
& \text{Haagerup property} \ar[d]^{(3)} \\
\text{Weak amenability with constant 1} \ar[r]^{(4)} \ar[d]^{(5)}
& \text{Weak Haagerup property with constant 1}  \ar[d]^{(6)} \\
\text{Weak amenability} \ar[r]^{(7)} \ar[d]^{(8)}
& \text{Weak Haagerup property} \\
\text{Approximation Property (AP)}
}
$$
	
	\caption{Approximation properties}
	\label{fig:aps}
\end{figure}

Figure~\ref{fig:aps} displays the relations between the approximation properties mentioned so far. At the moment, all implications are known to be strict except for (3) and (6). In a forthcoming paper \cite{HK-WH-examples} by Haagerup and the author, implication (6) will be shown to be strict as well.

The study of approximation properties of groups has important applications in the theory of operator algebras due to the fact that the approximation properties have operator algebraic counterparts. The standard examples are nuclearity of \Cs-algebras and semidiscreteness of von Neumann algebras which correspond to amenability of groups in the sense that a discrete group is amenable if and only if its reduced group \Cs-algebra is nuclear if and only if its group von Neumann algebra is semidiscrete (see \cite[Theorem~2.6.8]{MR2391387}). Also weak amenability and the Haagerup property have operator algebra analogues (see \cite[Chapter~12]{MR2391387}). In the second part of the present paper we introduce a von Neumann algebraic analogue of the weak Haagerup property and the weak Haagerup constant (see Definition~\ref{defi:WH-VN}).

\section{Main results}

The main results of this paper concern hereditary properties of the weak Haagerup property for locally compact groups and von Neumann algebras. As applications we are able to provide many examples of groups and von Neumann algebras with the weak Haagerup property. We additionally provide some reformulations of the weak Haagerup property (see Proposition~\ref{prop:equi-def} and Proposition~\ref{prop:equivalent-WH1}).

See Definition~\ref{defi:WH} for the definition of the weak Haagerup property for locally compact groups. Concerning the weak Haagerup property for locally compact groups we prove the following collection of hereditary results in Section~\ref{sec:hereditary-1}.

\begin{thmA}
Let $G$ be a locally compact group.
\begin{enumerate}
	\item If $H$ is a closed subgroup of $G$, and $G$ has the weak Haagerup property, then $H$ has the weak Haagerup property. More precisely, $$\Lambda_\WH(H) \leq \Lambda_\WH(G).$$
	\item If $K$ is a compact normal subgroup of $G$, then $G$ has the weak Haagerup property if and only if $G/K$ has the weak Haagerup property. More precisely, $$\Lambda_\WH(G) = \Lambda_\WH(G/K).$$
	\item The weak Haagerup property is preserved under finite direct products. More precisely, if $G'$ is a locally compact group, then $$\Lambda_\WH(G\times G') \leq \Lambda_\WH(G)\Lambda_\WH(G').$$
	\item If $(G_i)_{i\in I}$ is a directed set of open subgroups of $G$, then $$\Lambda_\WH(\bigcup_i G_i) = \lim_i \Lambda_\WH(G_i).$$
	\item If
	$
	1 \longrightarrow N \hooklongrightarrow G \longrightarrow G/N \longrightarrow 1
	$
	is a short exact sequence of locally compact groups, where $G$ is either second countable or discrete, and if $G/N$ is amenable, then $G$ has the weak Haagerup property if and only if $N$ has the weak Haagerup property. More precisely, $$\Lambda_\WH(G) = \Lambda_\WH(N).$$
	\item If $\Gamma$ is a lattice in $G$ and if $G$ is second countable, then $G$ has the weak Haagerup property if and only if $\Gamma$ has the weak Haagerup property. More precisely,
	$$
	\Lambda_\WH(\Gamma) = \Lambda_\WH(G).
	$$
\end{enumerate}
\end{thmA}

As mentioned, examples of groups with the weak Haagerup property trivially include all weakly amenable groups and groups with the Haagerup property. Apart from all these examples, we provide an additional example in Corollary~\ref{cor:class} to show that the class of weakly Haagerup groups is strictly larger than the class of weakly amenable groups and groups with the Haagerup property combined. Examples of groups without the weak Haagerup property will be one of the subjects of another paper \cite{HK-WH-examples} by Haagerup and the author.

See Definition~\ref{defi:WH-VN} and Remark~\ref{rem:independent-trace} for the definition of the weak Haagerup property for finite von Neumann algebras. Concerning the weak Haagerup property for finite von Neumann algebras we will prove the following theorems.

\begin{thmA}\label{thm:vna}
Let $\Gamma$ be a discrete group. The following conditions are equivalent.
\begin{enumerate}
	\item The group $\Gamma$ has the weak Haagerup property.
	\item The group von Neumann algebra $L(\Gamma)$ has the weak Haagerup property.
\end{enumerate}
More precisely, $\Lambda_\WH(\Gamma) = \Lambda_\WH(L(\Gamma))$.
\end{thmA}

\begin{thmA}\label{thm:hereditary2}
Let $M, M_1, M_2,\ldots$ be finite von Neumann algebras which admit faithful normal traces.

\begin{enumerate}
	\item If $M_2\subseteq M_1$ is a von Neumann subalgebra, then $\Lambda_\WH(M_2)\leq\Lambda_\WH(M_1)$.
	\item If $p\in M$ is a non-zero projection, then $\Lambda_\WH(pMp) \leq\Lambda_\WH(M)$.
	\item Suppose that $1\in M_1\subseteq M_2\subseteq \cdots$ are von Neumann subalgebras of $M$ generating all of $M$, and there is an increasing sequence of non-zero projections $p_n\in M_n$ with strong limit $1$. Then $\Lambda_\WH(M) = \sup_n \Lambda_\WH(p_nM_np_n)$.
	\item 
	$$
	\Lambda_\WH\left(\bigoplus_n M_n \right) = \sup_n \Lambda_\WH(M_n).
	$$
	\item
	$$
	\Lambda_\WH(M_1\tensor M_2) \leq \Lambda_\WH(M_1)\Lambda_\WH(M_2).
	$$
\end{enumerate}
\end{thmA}

As an application of the theorems above, in Section~\ref{sec:application} we give an example of two von Neumann algebras, in fact $\II_1$ factors, which are distinguished by the weak Haagerup property, i.e. the two von Neumann algebras do not have the same weak Haagerup constant. None of the other approximation properties mentioned in the introduction (see Figure~\ref{fig:aps}), or more precisely the corresponding operator algebraic approximation properties, can distinguish the two factors (see Remark~\ref{rem:other-APs}).

As another application of Theorem~\ref{thm:hereditary2} (or rather Theorem~\ref{thm:hereditary2-2} in Section~\ref{sec:hereditary2}) we are able to prove that the weak Haagerup constant of a von Neumann algebra with a faithful normal trace does not depend on the choice of trace (see Proposition~\ref{prop:independent-trace}).

Although the following result is not proved in this paper, we would like to mention it here, because it gives a complete description of the weak Haagerup property for connected simple Lie groups.

\begin{thm*}[\cite{HK-WH-examples}]
A connected simple Lie group has the weak Haagerup property if and only if it has real rank zero or one.
\end{thm*}

\section{Preliminaries}\label{sec:prelim}
We always let $G$ denote a locally compact group equipped with left Haar measure. We always include the Hausdorff requirement whenever we discuss topological groups and spaces.

The space of continuous functions on $G$ (with complex values) is denoted $C(G)$. It contains the subspace $C_0(G)$ of continuous functions vanishing at infinity and the subspace $C_c(G)$ of compactly supported continuous functions. When $G$ is a Lie group, $C^\infty(G)$ denotes the space of smooth functions on $G$.

In the following we introduce the Fourier algebra $A(G)$, the group von Neumann algebra $L(G)$, the completely bounded Fourier multipliers $M_0A(G)$, the algebra of Herz-Schur multipliers $B_2(G)$ and its predual $Q(G)$. This is quite a mouthful, so we encourage you to take a deep breath before you read any further. The most important of these spaces in the present context is the space of Herz-Schur multipliers $B_2(G)$ which occurs also in the definition of the weak Haagerup property, Definition~\ref{defi:WH}.

When $\pi$ is a continuous unitary representation of $G$ on some Hilbert space $\mc H$, and when $h,k\in\mc H$, then the continuous function $u$ defined by
\begin{align}\label{eq:matrix-coefficient}
u(x) = \la \pi(x)h,k\ra \quad \text{for all } x\in G
\end{align}
is a \emph{matrix coefficient} of $\pi$. The \emph{Fourier algebra} $A(G)$ is the space of matrix coefficients of the left regular representation $\lambda : G\to L^2(G)$. That is, $u\in A(G)$ if and only if there are $h,k\in L^2(G)$ such that
\begin{align}\label{eq:matrix-coefficient-lambda}
u(x) = \la \lambda(x)h,k\ra, \quad \text{for all } x\in G.
\end{align}
With pointwise operations, $A(G)$ becomes an algebra, and when equipped with the norm
$$
\| u\|_A = \inf \{ \|h\|_2 \|k\|_2 \mid \eqref{eq:matrix-coefficient-lambda} \text{ holds} \}.
$$
$A(G)$ is in fact a Banach algebra.

Given $u\in A(G)$ there are $f,g\in L^2(G)$ such that $u = f * \check g$ and $\|u\| = \|f\|_2\|g\|_2$, where $\check g(x) = g(x^{-1})$ and $*$ denotes convolution. This is often written as
$$
A(G) = L^2(G) * L^2(G).
$$
It is known that $\|u\|_\infty \leq \|u\|_A$ for any $u\in A(G)$, and $A(G) \subseteq C_0(G)$.

The Fourier algebra was introduced and studied in Eymard's excellent paper \cite{MR0228628} to which we refer to details about the Fourier algebra. When $G$ is not compact, the Fourier algebra $A(G)$ contains no unit. But it was shown in \cite{MR0239002} that $A(G)$ has a bounded approximate unit if and only if $G$ is amenable (see also \cite[Theorem~10.4]{MR767264}).

The von Neumann algebra generated by the image of the left regular representation $\lambda : G\to B(L^2(G))$ is the \emph{group von Neumann algebra}, $L(G)$. The Fourier algebra $A(G)$ can be identified isometrically with the (unique) predual of $L(G)$, where the duality is given by
$$
\la u,\lambda(x)\ra = u(x), \quad x\in G, \ u\in A(G).
$$

A function $v:G\to\C$ is called a \emph{Fourier multiplier}, if $vu \in A(G)$ for every $u\in A(G)$. A Fourier multiplier $v$ is continuous and bounded, and it defines bounded multiplication operator $m_v : A(G) \to A(G)$. The dual operator of $m_v$ is a normal (i.e. ultraweakly continuous) bounded operator $M_v : L(G)\to L(G)$ such that
$$
M_v \lambda(x) = v(x)\lambda(x).
$$
In \cite[Proposition~1.2]{MR784292} it is shown that Fourier multipliers can actually be characterized as the continuous functions $v:G\to\C$ such that
$$
\lambda(x) \mapsto v(x)\lambda(x)
$$
extends to a normal, bounded operator on the group von Neumann algebra $L(G)$. If $M_v$ is not only bounded but a completely bounded operator on $L(G)$, we say that $v$ is a \emph{completely bounded Fourier multiplier}. We denote the space of completely bounded Fourier multipliers by $M_0A(G)$. When equipped with the norm $\|v\|_{M_0A} = \|M_v\|_{\cb}$, where $\|\ \|_\cb$ denotes the completely bounded norm, $M_0A(G)$ is a Banach algebra. It is clear that
\begin{align}\label{eq:multiplier-norm}
\| v u \|_A \leq \|v \|_{M_0A} \|u\|_A \quad\text{ for every } v\in M_0A(G),\ u\in A(G).
\end{align}

One of the key notions of this paper is the notion of a Herz-Schur multiplier, which we now recall. Let $X$ be a non-empty set. A function $k:X\times X\to\C$ is called a \emph{Schur multiplier} on $X$ if for every bounded operator $A = [a_{xy}]_{x,y\in X} \in B(\ell^2(X))$ the matrix $[k(x,y)a_{xy}]_{x,y\in X}$ represents a bounded operator on $\ell^2(X)$, denoted $m_k(A)$. If $k$ is a Schur multiplier, it is a consequence of the closed graph theorem that $m_k$ defines a \emph{bounded} operator on $B(\ell^2(X))$. We define the \emph{Schur norm} $\|k\|_S$ to be the operator norm $\|m_k\|$ of $m_k$.

Let $u:G\to\C$ be a continuous function. Then $u$ is as \emph{Herz-Schur multiplier} if and only if the function $\hat u:G\times G\to\C$ defined by
$$
\hat u(x,y) = u(y^{-1}x), \quad x,y\in G,
$$
is a Schur multiplier on $G$. The set of Herz-Schur multipliers on $G$ is denoted $B_2(G)$. It is a Banach space, in fact a unital Banach algebra, when equipped with the \emph{Herz-Schur norm} $\|u\|_{B_2} = \|\hat u\|_S = \| m_{\hat u}\|$.

It is known that $B_2(G) = M_0A(G)$ isometrically (see \cite{MR753889}, \cite{MR1180643}, \cite[Theorem~5.1]{MR1818047}). We include several well-known characterizations of the Herz-Schur multipliers $B_2(G)$ below.

\begin{prop}\label{prop:Gilbert}
Let $G$ be a locally compact group, let $u:G\to\C$ be a function, and let $k \geq 0$ be given. The following are equivalent.
\begin{enumerate}
	\item $u$ is a Herz-Schur multiplier with $\|u\|_{B_2} \leq k$.
	\item $u$ is continuous, and for every $n\in\N$ and $x_1,\ldots,x_n\in G$
	$$
	\| u(x_j^{-1}x_i)_{i,j=1}^n \|_S \leq k.
	$$
	\item $u$ is a completely bounded Fourier multiplier with $\|u\|_{M_0A(G)} \leq k$.
	\item There exist a Hilbert space $\mc H$ and two bounded, continuous maps $P,Q:G\to\mc H$ such that
	$$
	u(y^{-1}x) = \la P(x),Q(y) \ra \qquad\text{for all } x,y\in G
	$$
	and
	$$
	(\sup_{x\in G} \|P(x)\|)(\sup_{y\in G} \|Q(y)\|) \leq k.
	$$
\end{enumerate}
If $G$ is second countable, then the above is equivalent to	
\begin{enumerate}
	\item[(5)] There exist a Hilbert space $\mc H$ and two bounded, Borel maps $P,Q:G\to\mc H$ such that
	$$
	u(y^{-1}x) = \la P(x),Q(y) \ra \qquad\text{for all } x,y\in G
	$$
	and
	$$
	(\sup_{x\in G} \|P(x)\|)(\sup_{y\in G} \|Q(y)\|) \leq k.
	$$
\end{enumerate}
\end{prop}

A proof taken from the unpublished manuscript \cite{UH-preprint} of the equivalence of (4) and (5) is included in the appendix (see Lemma~\ref{lem:HS-continuity}).

The space $B_2(G)$ of Herz-Schur multipliers has a Banach space predual. More precisely, let $Q(G)$ denote the completion of $L^1(G)$ in the norm
$$
\| f \|_Q = \sup \left\{ \left| \int_G f(x)u(x)\ dx \right| \mid u\in B_2(G),\ \|u\|_{B_2}\leq 1\} \right\}.
$$
In \cite{MR784292} it is proved that the dual Banach space of $Q(G)$ may be identified isometrically with $B_2(G)$, where the duality is given by
$$
\la f , u \ra = \int_G f(x)u(x)\ dx, \quad f\in L^1(G),\ u\in B_2(G).
$$
Thus, $B_2(G)$ may be equipped with the weak$^*$-topology arising from its predual $Q(G)$. This topology will also be denoted the $\sigma(B_2,Q)$-topology.

We note that since $\|u\|_\infty \leq \|u\|_{B_2}$ for any $u\in B_2(G)$, then $\|f\|_Q \leq \|f\|_1$ for every $f\in L^1(G)$. In particular, $C_c(G)$ is dense in $Q(G)$ with respect to the $Q$-norm, because $C_c(G)$ is dense in $L^1(G)$ with respect to the $1$-norm.

The Approximation Property (AP) briefly mentioned in the introduction is defined as follows. A locally compact group $G$ has AP if there is a net $(u_\alpha)$ in $A(G)$ such that $u_\alpha \to 1$ in the $\sigma(B_2,Q)$-topology. It was shown in \cite[Theorem~1.12]{MR1220905} that weakly amenable groups have AP. Only recently (in \cite{MR3047470}, \cite{delaat2}, \cite{MR2838352}) it was proved that there are (m)any groups without AP. Examples of groups without AP include the special linear groups $\SL_n(\R)$ when $n\geq 3$ and their lattices $\SL_n(\Z)$.

\section{The weak Haagerup property for locally compact groups}

The following definition is the main focus of the present paper.

\begin{defi}\label{defi:WH}
Let $G$ be a locally compact group. Then $G$ has the \emph{weak Haagerup property}, if there are a constant $C > 0$ and a net $(u_\alpha)_{\alpha\in A}$ in $B_2(G) \cap C_0(G)$ such that
$$
\|u_\alpha\|_{B_2} \leq C  \quad\text{ for every } \alpha\in A,
$$
$$
u_\alpha\to 1 \text{ uniformly on compacts as } \alpha\to\infty.
$$

The weak Haagerup constant $\Lambda_{\WH}(G)$ is defined as the infimum of those $C$ for which such a net $(u_\alpha)$ exists, and if no such net exists we write $\Lambda_{\WH}(G) = \infty$. It is not hard to see that the infimum is actually a minimum. If a group $G$ has the weak Haagerup property, we will also sometimes say that $G$ \emph{is weakly Haagerup}.
\end{defi}

If, in the above definition, ones replaces the requirement $u_\alpha\in C_0(G)$ with the stronger requirement $u_\alpha\in C_c(G)$, one obtains the definition of weak amenability.

Apart from the norm topology, there are (at least) three interesting topologies one can put on the norm bounded sets in $B_2(G)$ one of which is the locally uniform topology used in Definition~\ref{defi:WH} and the others being the $\sigma(B_2,Q)$-topology and the point-norm topology (see Appendix~\ref{app:topologies}). Proposition~\ref{prop:weak-conv} and \ref{prop:equi-def} below show that any of these three topologies could have been used in Definition~\ref{defi:WH}. More precisely, we have the following characterizations of the weak Haagerup property.

\begin{prop}\label{prop:weak-conv}
Let $G$ be a locally compact group. Then $\Lambda_\WH(G) \leq C$ if and only if there is a net $(u_\alpha)$ in $B_2(G) \cap C_0(G)$ such that
$$
\|u_\alpha\|_{B_2} \leq C  \quad\text{ for every } \alpha,
$$
$$
u_\alpha\to 1 \text{ in the } \sigma(B_2,Q)\text{-topology}.
$$
\end{prop}
\begin{proof}
Suppose first $\Lambda_\WH(G) \leq C$. Then by Lemma~\ref{lem:A1}~(2), the conditions in our proposition are satisfied.

Conversely, suppose we are given a net $(u_\alpha)$ in $B_2(G) \cap C_0(G)$ such that
$$
\|u_\alpha\|_{B_2} \leq C  \quad\text{ for every } \alpha,
$$
$$
u_\alpha\to 1 \text{ in the } \sigma(B_2,Q)\text{-topology}.
$$
Let $v_\alpha = h * u_\alpha$, where $h$ is a continuous, non-negative, compactly supported function on $G$ such that $\int h(x)\ dx = 1$. Then using the convolution trick (see Lemma~\ref{lem:B1}, Lemma~\ref{lem:B2} and Remark~\ref{rem:conv-trick}) we see that the net $(v_\alpha)$ witnesses $\Lambda_\WH(G) \leq C$.
\end{proof}

The following Proposition (and its proof) is inspired by \cite[Proposition~1.1]{MR996553}.
\begin{prop}\label{prop:equi-def}
Let $G$ be a locally compact group and suppose $\Lambda_\WH(G) \leq C$. Then there exists a net $(v_\alpha)_{\alpha\in A}$ in $B_2(G) \cap C_0(G)$ such that
$$
\|v_\alpha\|_{B_2} \leq C  \quad\text{ for every } \alpha,
$$
$$
\|v_\alpha u - u\|_A \to 0 \quad\text{ for every } u\in A(G),
$$
$$
v_\alpha\to 1 \text{ uniformly on compacts.}
$$

If $L$ is any compact subset of $G$ and $\epsilon > 0$, then there exists $w\in B_2(G)\cap C_0(G)$ so that
$$
\|w\|_{B_2} \leq C+\epsilon,
$$
$$
w = 1 \quad\text{ for every } x\in L.
$$
Moreover, if $K$ is a compact subgroup of $G$, then the net $(v_\alpha)$ and can be chosen to consist of $K$-bi-invariant functions.
Finally, if $G$ is a Lie group, the net $(v_\alpha)$ can additionally be chosen to consist of smooth functions.
\end{prop}
\begin{proof}
Let $(u_\alpha)$ be a net witnessing $\Lambda_\WH(G) \leq C$. Using the bi-invariance trick (see Appendix~\ref{app:average}) we see that the net $(u_\alpha^K)$ obtained by averaging each $u_\alpha$ from left and right over the compact subgroup $K$ is a net of $K$-bi-invariant functions witnessing $\Lambda_\WH(G)\leq C$. We let $v_\alpha = h^K * u_\alpha^K$, where $h \in C_c(G)$ is a non-negative, continuous function with compact support and integral 1. Using the convolution trick (see Lemma~\ref{lem:B1} and Lemma~\ref{lem:B2}) we see that the net $(v_\alpha)$ has the desired properties (that $v_\alpha\to 1$ uniformly on compacts follows from Lemma~\ref{lem:A1}).

Let $L\subseteq G$ be compact and $\epsilon > 0$ be arbitrary. By \cite[Lemma~3.2]{MR0228628} there is $u\in A(G)$ such that $u(x) = 1$ for all $x\in K$. According to the first part of our proposition, there is $v\in B_2(G)\cap C_0(G)$ such that $\|v\|_{B_2} \leq C$ and $\|vu - u \|_A \leq \epsilon$. Let $w = v - (vu - u)$. Then $w$ has the desired properties.

If $G$ is a Lie group, we let $h$ be as before with the extra condition that $C^\infty(G)$ and use the arguments above.
\end{proof}

Proposition~\ref{prop:equivalent-WH1} gives an equivalent formulation of the weak Haagerup property with constant 1. Recall that a continuous map is \emph{proper}, if the preimage of a compact set is compact.

\begin{prop}\label{prop:equivalent-WH1}
Let $G$ be a locally compact and $\sigma$-compact group. Then $G$ is weakly Haagerup with constant 1, if and only if there is a continuous, proper function $\psi:G\to [0,\infty[$ such that $\|e^{-t\psi}\|_{B_2} \leq 1$ for every $t>0$.

Moreover, we can take $\psi$ to be symmetric.
\end{prop}
The idea to the proof of the proposition is taken from the proof of Proposition 2.1.1 in \cite{MR1852148}. A proof in the case where $G$ is discrete can be found in \cite{MR3146826}.
\begin{proof}
Suppose first such a map $\psi$ exists, and let $u_t = e^{-t\psi}$. The fact that $\psi$ is proper implies that $u_t \in C_0(G)$ for every $t>0$. If $K\subseteq G$ is compact, then $\psi(K) \subseteq [0,r]$ for some $r > 0$. Hence $u_t(K) \subseteq [e^{-tr},1]$. This shows that $u_t \to 1$ uniformly on $K$ as $t\to 0$. It follows that $G$ is weakly Haagerup with constant 1.

Conversely, suppose $G$ is weakly Haagerup with constant 1. Since $G$ is locally compact and $\sigma$-compact, it is the union of an increasing sequence $(U_n)_{n=1}^\infty$ of open sets such that the closure $\overline{U_n}$ of $U_n$ is compact and contained in $U_{n+1}$ (see \cite[Proposition~4.39]{MR1681462}). Choose an increasing, unbounded sequence $(\alpha_n)$ of positive real numbers and a decreasing sequence $(\epsilon_n)$ tending to zero such that $\sum_n \alpha_n\epsilon_n$ is finite. For every $n$ choose a function $u_n \in B_2(G) \cap C_0(G)$ with $\|u_n\|_{B_2} \leq 1$ such that
$$
\sup_{g\in \overline{U_n}} |u_n(g) - 1| \leq \epsilon_n/2.
$$
Replace $u_n$ by $|u_n|^2$, if necessary, to ensure $0 \leq u_n \leq 1$ and
$$
\sup_{g\in \overline{U_n}} |u_n(g) - 1| \leq \epsilon_n.
$$
Define $\psi_i:G\to[0,\infty[$ and $\psi:G\to[0,\infty[$ by
$$
\psi_i(g) = \sum_{n=1}^i \alpha_n (1-u_n(g)), \qquad \psi(g) = \sum_{n=1}^\infty \alpha_n (1-u_n(g)).
$$
It is easy to see that $\psi$ is well-defined. We claim that $\psi_i\to \psi$ uniformly on compacts. For this, let $K\subseteq G$ be compact. By compactness, $K \subseteq U_N$ for some $N$, and hence if $g\in K$ and $i\geq N$,
$$
|\psi(g) - \psi_i(g)| = |\sum_{n=i+1}^\infty \alpha_n (1 - u_n(g)) | \leq \sum_{n=i+1}^\infty \alpha_n \epsilon_n.
$$
Since $\sum_n \alpha_n\epsilon_n$ converges, this proves that $\psi_i\to\psi$ uniformly on $K$. In particular, since each $\psi_i$ is continuous, $\psi$ is continuous.

We claim that $\psi$ is proper. Let $R > 0$ be given, and choose $n$ such that $\alpha_n \geq 2R$. Since $u_n \in C_0(G)$, there is a compact set $K \subseteq G$ such that $|u_n(g)| < 1/2$ whenever $g\in G\setminus K$. Now if $\psi(g) \leq R$, then $\psi(g) \leq \alpha_n / 2$, and in particular $\alpha_n (1-u_n(g)) \leq \alpha_n/2$, which implies that $1-u_n(g) \leq 1/2$. Hence we have argued that
$$
\{g\in G \mid \psi(g) \leq R \} \subseteq \{g\in G \mid 1-u_n(g) \leq 1/2 \}\subseteq K.
$$
This proves that $\psi$ is proper.

Now let $t>0$ be fixed. We must show that $\|e^{-t\psi}\|_{B_2} \leq 1$. 
Since $\psi_i$ converges locally uniformly to $\psi$, it will suffice to prove that $\|e^{-t\psi_i}\|_{B_2} \leq 1$, because the unit ball of $B_2(G)$ is closed under locally uniform limits (see Lemma~\ref{lem:B2-unitball}). Observe that
$$
e^{-t\psi_i} = \prod_{n=1}^i e^{-t \alpha_n (1-u_n) },
$$
and so it suffices that $e^{-t \alpha_n (1-u_n) }$ belongs to the unit ball of $B_2(G)$ for each $n$. And this is clear, since
$$
\| e^{-t \alpha_n (1-u_n) } \|_{B_2} = e^{-t \alpha_n} \| e^{t \alpha_n u_n } \|_{B_2} \leq e^{-t \alpha_n}  e^{t \alpha_n \| u_n\|_{B_2} } \leq 1.
$$

To prove the last assertion, put $\bar\psi = \psi + \check\psi$, where $\check\psi(g) = \psi(g^{-1})$. Clearly, $\bar\psi$ is continuous and proper. Finally, for every $t>0$
$$
\|e^{-t\bar\psi}\|_{B_2} \leq \|e^{-t\psi}\|_{B_2} \|e^{-t\check\psi}\|_{B_2} \leq 1,
$$
since $\|\check u\|_{B_2} = \|u\|_{B_2}$ for every Herz-Schur multiplier $u\in B_2(G)$.

\end{proof}

Having settled the definition of the weak Haagerup property for locally compact groups and various reformulations of the property, we move on to prove hereditary results for the class of groups with the weak Haagerup property.

\section{Hereditary properties I}\label{sec:hereditary-1}
In this section we prove hereditary results for the weak Haagerup property of locally compact groups. The hereditary properties under consideration involve passing to closed subgroups, taking quotients by compact normal subgroups, taking finite direct products, taking direct unions of open subgroups and extending from co-F{\o}lner subgroups and lattices to the whole group.

We begin this section with an easy lemma.
\begin{lem}
Suppose $G$ is a locally compact group with a closed subgroup $H$.
\begin{enumerate}
	\item If $u \in C_0(G)$, then $u|_H \in C_0(H)$.
	\item If $u \in B_2(G)$, then $u|_H \in B_2(H)$ and
	$$
	\| u|_H \|_{B_2(H)} \leq \|u\|_{B_2(G)}.
	$$
 \end{enumerate}
\end{lem}

\begin{proof}
(1) is obvious, and (2) is obvious from the characterization in Proposition~\ref{prop:Gilbert}.
\end{proof}

An immediate consequence of the previous lemma is the following.

\begin{prop}\label{prop:subgroup}
The class of weakly Haagerup groups is stable under taking subgroups. More precisely, if $G$ is a locally compact group with a closed subgroup $H$, then $\Lambda_{\WH}(H) \leq \Lambda_{\WH}(G)$.
\end{prop}

\begin{lem}\label{lem:K-quotient}
If $K\subseteq G$ is a compact, normal subgroup, then
\begin{enumerate}
  \item $C(G/K)$ may be canonically and isometrically identified with the subspace of $C(G)$ of functions constant on the cosets of $K$ in $G$.
	\item Under the canonical identification from (1), $C_0(G/K)$ is isometrically identified with the subspace of $C_0(G)$ of functions constant on the cosets of $K$ in $G$.
	\item Under the canonical identification from (1), $B_2(G/K)$ is isometrically identified with the subspace of $B_2(G)$ of functions constant on the cosets of $K$ in $G$.
	\item Moreover, the canonical identification preserves the topology of locally uniform convergence.
\end{enumerate}
\end{lem}
\begin{proof}
\mbox{}

(1) Let $q:G\to G/K$ denote the quotient map. If $f\in C(G)$ is constant on $K$-cosets, it is easy to see that the induced map $\bar f$ defined by $\bar f([x]_K) = f(x)$ is continuous. Conversely, if $g\in C(G/K)$ is given, then the composite $g\circ q$ is continuous on $G$ and constant on cosets.

(2) One must check that $g\in C_0(G/K)$ if and only if $g\circ q\in C_0(G)$. Note first that a subset $L \subseteq G/K$ is compact if and only if $q^{-1}(L)$ is compact. In other words, $q$ is proper. The rest is elementary. It is also clear, that the correspondence is isometric with respect to the uniform norm. This completes (2).

(3) This is Proposition 1.3 in \cite{MR996553}.

(4) One must check that if $(g_n)$ is a net in $C(G/K)$ and $g\in C(G/K)$, then $g_n\to g$ uniformly on compacts if and only if $g_n \circ q \to g\circ q$ uniformly on compacts. This is elementary using properness of $q$.
\end{proof}

\begin{prop}\label{prop:modK}
If $G$ is a locally compact group with a compact, normal subgroup $K\triangleleft G$, then $\Lambda_{\WH}(G/K) = \Lambda_{\WH}(G)$.
\end{prop}
\begin{proof}

Apply the last part of Proposition~\ref{prop:equi-def} and Lemma~\ref{lem:K-quotient}.
\end{proof}

Concerning direct products of groups we have the following proposition.

\begin{prop}\label{prop:products}
The class of weakly Haagerup groups is stable under finite direct products. More precisely, we have
\begin{align}\label{eq:prod-ineq}
\Lambda_\WH(G\times H) \leq \Lambda_\WH(G)\Lambda_\WH(H)
\end{align}
for locally compact groups $G$ and $H$.
\end{prop}
\begin{proof}
From the characterization in Proposition~\ref{prop:Gilbert}, it easily follows that if $u\in B_2(G)$ and $v\in B_2(H)$, then $u\times v \in B_2(G\times H)$ and $\| u\times v\|_{B_2} \leq \|u\|_{B_2} \|v\|_{B_2}$. Also, if  $u\in C_0(G)$ and $v\in C_0(H)$, then $u\times v \in C_0(G\times H)$: If $\epsilon > 0$ is given, then there are compact subsets $L_1 \subseteq G$, $L_2\subseteq H$ such that
$$
|u(g)| < \epsilon / \|v\|_\infty  \quad \text{ when } g\notin L_1, \quad\text{ and }\quad |v(h)| < \epsilon / \|u\|_\infty \quad \text{ when } h\notin L_2.
$$
(of course, we assume that $\|u\|_\infty\|v\|_\infty \neq 0$). If $(g,h) \notin L_1 \times L_2$, say $h\notin L_2$, then
$$
|u(g)v(h)| < \|u\|_\infty \ \epsilon / \|u\|_\infty = \epsilon.
$$
It is now clear that if $(u_\alpha)$ and $(v_\beta)$ are bounded nets in $B_2(G)\cap C_0(G)$ and $B_2(H)\cap C_0(H)$, respectively, converging locally uniformly to $1$, then the product net $(u_\alpha \times v_\beta)$ (with the product order) belongs to $B_2(G\times H) \cap C_0(G\times H)$ and converges locally uniformly to $1$. This proves that
$$
\Lambda_\WH(G\times H) \leq \Lambda_\WH(G)\Lambda_\WH(H).
$$
\end{proof}

\begin{rem}\label{rem:products}
It would of course be interesting to know if equality actually holds in \eqref{eq:prod-ineq}. The corresponding result for weak amenability is known to be true (see \cite[Corollary~1.5]{MR996553}). It is not hard to see that if either $\Lambda_\WH(G) = 1$ or $\Lambda_\WH(H) = 1$, then \eqref{eq:prod-ineq} is an equality.
\end{rem}

With Proposition~\ref{prop:products} at our disposal, we can show the following.

\begin{cor}\label{cor:class}
The class of weakly Haagerup groups contains groups that are neither weakly amenable nor have the Haagerup property.
\end{cor}
\begin{proof}
It is known that the Lie group $G = \Sp(1,n)$ is weakly amenable with $\Lambda_\WA(G) = 2n - 1$ (see \cite{MR996553}). It is also known that $G$ has Property (T) when $n\geq 2$ (see \cite[Section~3.3]{MR2415834}), and hence $G$ does not have the Haagerup property (since $G$ is not compact). As we mentioned earlier, the group $H = \Z/2 \wr \F_2$ has the Haagerup property, but is not weakly amenable. Hence both $G$ and $H$ have the weak Haagerup property. It now follows from the previous proposition that the group $G\times H$ has the weak Haagerup property.

Both the Haagerup property and weak amenability passes to subgroups, so it also follows that $G\times H$ has neither of these properties. 
\end{proof}

\begin{rem}
If you want an example of a \emph{discrete} group with the weak Haagerup property outside the class of weakly amenable groups and the Haagerup groups, then take $\Gamma$ to be a lattice in $\Sp(1,n)$ and consider the group $\Gamma\times H$, where again $H = \Z/2\wr\F_2$.

The group constructed in the proof of Corollary~\ref{cor:class} is of course tailored exactly to prove the corollary, and one might argue that it is not a natural example. It would be interesting to find more natural examples, for instance a simple group.
\end{rem}

Using the characterization of Herz-Schur multipliers given in Proposition~\ref{prop:Gilbert}, it is not hard to prove the following (see \cite[Lemma~4.2]{MR1373647}).

\begin{lem}\label{lem:HS-extension}
Let $H$ be an open subgroup of a locally compact group $G$. Extend $u\in B_2(H)$ to $\tilde u : G\to\C$ by letting $\tilde u(x) = 0$ when $x\notin H$. Then $\tilde u\in B_2(G)$ and $\|u\|_{B_2} = \|\tilde u\|_{B_2}$.

Moreover, if $ u\in C_0(H)$, then $\tilde u\in C_0(G)$.
\end{lem}

We note that there are examples of groups $H\leq G$, where some $ u\in B_2(H)$ has no extension to $B_2(G)$ (see \cite[Theorem~4.4]{Brannan-Forrest}). In these examples, $H$ is of course not open.

\begin{prop}
If $(G_i)_{i\in I}$ is a directed set of open subgroups in a locally compact group $G$, and $G = \bigcup_i G_i$, then
$$
\Lambda_\WH(G) = \sup_i \Lambda_\WH(G_i).
$$
\end{prop}
\begin{proof}
From Proposition~\ref{prop:subgroup} we already know that $\Lambda_\WH(G) \geq \sup_i \Lambda_\WH(G_i)$. We will now show the other inequality. We may assume that $\sup_i \Lambda_\WH(G_i) < \infty$ since otherwise there is nothing to prove.

Let $L\subseteq G$ be a compact set and let $\epsilon > 0$ be given. By compactness and directedness there is $i\in I$ such that $L\subseteq G_i$. Using Proposition~\ref{prop:equi-def} we may find $w\in B_2(G_i) \cap C_0(G_i)$ so that
$$
\|w\|_{B_2} \leq \Lambda_\WH(G_i)+\epsilon \leq \sup_i \Lambda_\WH(G_i) + \epsilon,
$$
$$
w(x) = 1 \quad\text{ for every } x\in L.
$$
By Lemma~\ref{lem:HS-extension}, there is $\tilde w \in B_2(G) \cap C_0(G)$ such that
$$
\|\tilde w\|_{B_2} \leq \sup_i \Lambda_\WH(G_i)+\epsilon,
$$
$$
\tilde w(x) = 1 \quad\text{ for every } x\in L.
$$
Since $L$ and $\epsilon$ were arbitrary, it now follows that
$$
\Lambda_\WH(G) \leq \sup_i \Lambda_\WH(G_i),
$$
and the proof is complete.
\end{proof}

The next result, Proposition~\ref{prop:proper-cocycle}, is inspired by \cite{MR1756981}. Let $G$ be a locally compact, second countable group, and let $(X,\mu)$ be a standard measure space with a Borel action of $G$. We assume that the measure $\mu$ is a probability measure which is invariant under the action. In \cite{MR1756981}, quasi-invariant measures are considered as well, but we will stick to invariant measures all the time, because the invariance is needed in the proof of Lemma~\ref{lem:inducing-properties} (1) and (3).

Further, let $H$ be a locally compact, second countable group, and let $\alpha:G\times X \to H$ be a Borel cocycle, i.e. $\alpha$ is a Borel map and for all $g,h\in G$ we have
$$
\alpha(gh,x) = \alpha(g,hx)\alpha(h,x) \qquad\text{for } \mu\text{-almost all } x\in X.
$$

The following definition of a proper cocycle is taken from \cite{MR1756981}, although we have modified it slightly.
\begin{defi}\label{defi:proper-cocycle}
Let $\alpha:G\times X\to H$ be as above. We say that $\alpha$ is \emph{proper}, if there is a generating family $\mc A$ of Borel subsets of $X$ such that the following three conditions hold.
\begin{enumerate}
	\item $X$ is the union of an increasing sequence of elements in $\mc A$.
	\item For every $A\in\mc A$ and every compact subset $L$ of $G$ the set $\alpha(L\times A)$ is pre-compact.
	\item For every $A\in\mc A$ and every compact subset $L$ of $H$, the set $K(A,L)$ of elements $g\in G$ such that $\{ x \in A\cap g^{-1}A \mid \alpha(g,x) \in L\}$ has positive $\mu$-measure is pre-compact.
\end{enumerate}
\end{defi}

We mention the following examples of proper cocycles. All examples are taken from \cite[p. 490]{MR1756981}. 

\begin{exam}\label{exam:proper-cocycle}
\mbox{}

\begin{enumerate}
	\item[(a)] Suppose $H$ is a closed subgroup of $G$ and that $X = G/H$ has an invariant probability measure $\mu$ for the action of left translation. Let $\sigma: G/H \to G$ be a regular Borel cross section of the projection map $p: G \to G/H$, i.e. a Borel map such that $p\circ\sigma = \id_{G/H}$ and $\sigma(L)$ has compact closure for each compact $L\subseteq G/H$ (see \cite[Lemma~1.1]{MR0044536}). We define $\alpha: G\times X\to H$ by
	$$
	\alpha(g,x) = \sigma(gx)^{-1} g\sigma(x).
	$$
	With $\mc A$ the family of all compact subsets of $X$, we verify the three conditions in Definition~\ref{defi:proper-cocycle}. Since $G$ is second countable, it is also $\sigma$-compact. Then $X$ is also $\sigma$-compact, and condition (1) is satisfied.
	
	Let $A\in\mc A$ and let $L\subseteq G$ be compact. By regularity of $\gamma$,
	$$
	\alpha(L\times A) \subseteq \sigma(LA)^{-1} L\sigma(A)
	$$
	is pre-compact, and condition (2) is satisfied.
	
	Let $A\in\mc A$ and let $L\subseteq H$ be compact. It is easy to see that
	$$
	K(A,L) \subseteq \sigma(A)L\sigma(A)^{-1}.
	$$
	Again by regularity of $\sigma$, it follows that $K(A,L)$ is pre-compact. Thus, condition (3) is satisfied.
	\item[(b)] Suppose $K\triangleleft G$ is normal and compact. Let $H = G/K$, let $X = K$ and let $\mu$ be the normalized Haar measure on $K$. Then $G$ acts on $K$ by conjugation, and $\mu$ is invariant under this action. We let $\mc A$ be the collection of all Borel subsets of $K$, and we define $\alpha:G\times X\to H$ by
	$$
	\alpha(g,x) = p(g),
	$$
	where $p: G \to H$ is the projection map. Conditions (1) and (2) of Definition~\ref{defi:proper-cocycle} are immediate. For condition (3) we first note that if $L\subseteq H$ is compact, then $\alpha^{-1}(L) = p^{-1}(L) \times K$. Since $p$ is a homomorphism with compact kernel, it is proper. Hence $p^{-1}(L)$ is compact, and $K(A,L)\subseteq p^{-1}(L)$.
\end{enumerate}
We emphasize the following special case of (a).
\begin{enumerate}
	\item[(c)] Recall that a subgroup $\Gamma \subseteq G$ is a \emph{lattice}, if $\Gamma$ is discrete and the quotient space $G/\Gamma$ admits a finite $G$-invariant measure. Hence, when $H = \Gamma$ is a lattice in $G$, we are in the situation mentioned in (a).
\end{enumerate}

\end{exam}

Let $G$ and $H$ be locally compact, second countable groups, and let $(X,\mu)$ be a standard $G$-space with a $G$-invariant probability measure. Let $\alpha:G\times X\to H$ be a proper Borel cocycle. When $ u\in B_2(H)$ we define $\hat u:G\to\C$ by
\begin{align}\label{eq:uhat}
\hat u(g) = \int_X  u(\alpha(g,x))\ d\mu(x), \qquad g\in G.
\end{align}
The construction is taken from \cite{MR1756981}, where it is shown in Lemma~2.11 that $\hat u \in B_2(G)$ and also $\|\hat u\|_{B_2} \leq \| u\|_{B_2}$. We refer to Lemma~\ref{lem:HS-continuity} for the continuity of $\hat u$.

\begin{lem}\label{lem:inducing-properties}
Let $\alpha:G\times X\to H$ be a proper cocycle as above, and let $u\in B_2(H)$ be given.
\begin{enumerate}
	\item $\hat u\in B_2(G)$ and $\|\hat u\|_{B_2} \leq \| u\|_{B_2}$.
	\item $\|\hat u\|_\infty \leq \| u\|_\infty$.
	\item If $ u\in C_0(H)$, then $\hat u\in C_0(G)$.
\end{enumerate}
\end{lem}
\begin{proof}
\mbox{}

(1) This is \cite[Lemma~2.11]{MR1756981}.

(2) This is obvious.

(3) Given $\epsilon > 0$ there is $L\subseteq H$ compact such that $h\notin L$ implies $| u(h)| \leq \epsilon$. Since $X$ is the union of an increasing sequence of sets in $\mc A$, we may take $A\in\mc A$ such that $\mu(X\setminus A) \leq \epsilon$. The set $K = \overline{K(A,L)}$ is compact in $G$, and if $g\notin K$ then
$$
X_g = \{x\in A\cap g^{-1}A \mid \alpha(g,x) \in L\}
$$
is a null set. Hence for $g\notin K$
\begin{align*}
|\hat u(g)| &\leq \int_{X\setminus X_g} | u(\alpha(g,x))| \ d\mu(x) \\
&\leq \int_{X\setminus (A\cap g^{-1}A)} \| u\|_\infty \ d\mu(x)  +  \int_{(A\cap g^{-1}A)\setminus X_g} \epsilon \ d\mu(x) \\
&\leq  2 \epsilon \|u\|_\infty + \epsilon.
\end{align*}
This shows that $\hat u\in C_0(G)$.
\end{proof}

\begin{lem}\label{lem:inducing-continuous}
Let $\alpha:G\times X\to H$ be a proper cocycle as above. The contractive linear map $B_2(H)\to B_2(G)$ defined by $u\mapsto \hat u$, where $\hat u$ is given by \eqref{eq:uhat}, is continuous on norm bounded sets with respect to the topology of locally uniform convergence.
\end{lem}
\begin{proof}
Suppose $ u_n\to 0$ in $B_2(H)$ uniformly on compacts, and $\| u_n\|_{B_2} < c$ for every $n$. In particular, $\| u_n\|_\infty < c$ for every $n$. Let $K\subseteq G$ be compact, and let $\epsilon > 0$ be given. Choose $A\in\mc A$ such that $\mu(X\setminus A) \leq \epsilon/2c$, and let $L = \overline{\alpha(K\times A)}$. Since $L$ is compact, we have eventually that $| u_n(h)| < \epsilon/2$ for every $h\in L$. Then for $g\in K$ we have
\begin{align*}
|\hat u_n(g)| = \left| \int_X  u_n(\alpha(g,x)) \ d\mu(x) \right| 
\leq \int_A \epsilon/2 \ d\mu(x) + \int_{X\setminus A} c \ d\mu(x) 
\leq \epsilon.
\end{align*}
This completes the proof.
\end{proof}

\begin{prop}\label{prop:proper-cocycle}
Let $G$ and $H$ be locally compact, second countable group, and let $(X,\mu)$ be a standard Borel $G$-space with a $G$-invariant probability measure. If there is a proper Borel cocycle $\alpha:G\times X\to H$, then $\Lambda_{\WH}(G) \leq \Lambda_{\WH}(H)$.
\end{prop}
\begin{proof}
Suppose $\Lambda_\WH(H) \leq C$, and choose a net $( u_i)$ in $B_2(H) \cap C_0(H)$ such that
$$
\| u_i\|_{B_2} \leq C  \quad\text{ for every } i,
$$
$$
 u_i\to 1 \text{ uniformly on compacts }.
$$
It follows from Lemma~\ref{lem:inducing-properties} that $\hat u_i\in B_2(G)\cap C_0(G)$ and
$$
\|\hat u_i\|_{B_2} \leq C  \quad\text{ for every } i.
$$
From Lemma~\ref{lem:inducing-continuous} we also see that
$$
\hat u_i\to 1 \text{ uniformly on compacts }.
$$
This shows that $\Lambda_\WH(G) \leq C$, and the proof is complete.
\end{proof}

In view of Example~\ref{exam:proper-cocycle} (a) we get the following corollary.
\begin{cor}
Let $G$ be a locally compact, second countable group with a closed subgroup $H$ such that $G/H$ admits a $G$-invariant probability measure. Then $G$ is weakly Haagerup if and only if $H$ is weakly Haagerup. More precisely, $\Lambda_{\WH}(G) = \Lambda_{\WH}(H)$.
\end{cor}
\begin{proof}
From Proposition~\ref{prop:subgroup} we know that $\Lambda_\WH(H)\leq \Lambda_\WH(G)$. The other inequality follows from Proposition~\ref{prop:proper-cocycle} in view of Example~\ref{exam:proper-cocycle}.
\end{proof}

\begin{cor}\label{cor:lattice}
Let $G$ be a locally compact, second countable group with a lattice $\Gamma \subseteq G$. Then $G$ is weakly Haagerup if and only if $\Gamma$ is weakly Haagerup. More precisely, $\Lambda_{\WH}(G) = \Lambda_{\WH}(\Gamma)$.
\end{cor}

Inspired by the proof of Proposition~\ref{prop:proper-cocycle} we now set out to prove that the weak Haagerup property can be lifted from a co-F{\o}lner subgroup to the whole group. In particular, extensions of amenable groups by weakly Haagerup groups yield weakly Haagerup groups.

Recall that a closed subgroup $H$ in a locally compact group $G$ is \emph{co-F{\o}lner} if there is a $G$-invariant Borel measure $\mu$ on the coset space $G/H$ and if for each $\epsilon > 0$ and compact set $L\subseteq G$ there is a compact set $F \subseteq G/H$ such that $0 < \mu(F) < \infty$ and
$$
\frac{\mu(gF \triangle F)}{\mu(F)} < \epsilon \quad\text{for all } g\in L.
$$
Here $\triangle$ denotes symmetric difference of sets. The most natural examples of co-F{\o}lner subgroups are closed normal subgroups with amenable quotients. Indeed, it follows from the F{\o}lner characterization of amenability (see \cite[Theorem~7.3]{MR767264} and \cite[Proposition~7.4]{MR767264}) that such groups are co-F{\o}lner.

\begin{prop}
Let $G$ be a locally compact group with a closed subgroup $H$. Assume that $G$ is either second countable or discrete. If $H$ is weakly Haagerup and co-F{\o}lner, then $G$ is weakly Haagerup. More precisely, $\Lambda_\WH(G) = \Lambda_\WH(H)$.
\end{prop}
\begin{proof}
Let $C = \Lambda_\WH(H)$. We already know from Proposition~\ref{prop:subgroup} that $\Lambda_\WH(G) \geq C$, so it suffices to prove the other inequality. For this it is enough prove that for each compact $L\subseteq G$ and $\epsilon > 0$ there is $v\in B_2(G) \cap C_0(G)$ with $\|v\|_{B_2} \leq C$ such that
$$
|v(g) - 1| \leq 2\epsilon \quad\text{for all } g\in L.
$$

Thus, suppose that $L\subseteq G$ is compact and $\epsilon > 0$. Let $\sigma: G/H\to G$ be a regular Borel cross section. If $G$ is discrete the existence of $\sigma$ is trivial, and if $G$ is second countable then the existence of $\sigma$ is a standard result (see \cite[Lemma~1.1]{MR0044536}). Define the corresponding cocycle $\alpha : G\times G/H \to H$ by
$$
\alpha(g,x) = \sigma(gx)^{-1}g\sigma(x) \quad\text{for all } g\in G,\ x\in G/H.
$$
Choose an invariant Borel measure $\mu$ on $G/H$ and a compact set $F\subseteq G/H$ such that $0 < \mu(F) < \infty$ and
$$
\frac{\mu(g F \triangle F)}{\mu(F)} < \epsilon \quad\text{for all } g\in L.
$$
By regularity of $\sigma$, the set $K = \overline{\alpha(L\times F)}$ is compact, because
$$
\alpha(L\times F) \subseteq \sigma(LF)^{-1} L \sigma(F).
$$
Since $\Lambda_\WH(H) \leq C$ there is a Herz-Schur multiplier $u \in B_2(H)\cap C_0(H)$ such that $\|u\|_{B_2}\leq C$ and
$$
|u(h) - 1| \leq \epsilon \quad\text{for all } h\in K.
$$
Define $v: G\to \C$ by
$$
v(g) = \frac{1}{\mu(F)} \int_{G/H} 1_{F\cap g^{-1} F} (x) u(\alpha(g,x)) \ d\mu(x).
$$
We claim that $v$ has the desired properties. First we check that $v\in B_2(G)$ with $\|v\|_{B_2} \leq C$. Since $u\in B_2(H)$ there are a Hilbert space $\mc H$ and bounded, continuous maps $P,Q: H\to \mc H$ such that
$$
u(ab^{-1}) = \la P(a),Q(b)\ra \quad\text{ for all } a,b\in H.
$$
If $G$ is second countable, then so is $H$ and we can (and will) assume that $\mc H$ is separable. Consider the Hilbert space $L^2(G/H,\mc H)$, and define Borel maps $\tilde P,\tilde Q:G\to L^2(G/H,\mc H)$ by
\begin{align*}
\tilde P(g)(x) &= \frac{1}{\mu(F)^{1/2}} 1_{g^{-1}F}(x) P(\alpha(g,x)) \\
\tilde Q(g)(x) &= \frac{1}{\mu(F)^{1/2}} 1_{g^{-1}F}(x) Q(\alpha(g,x))
\end{align*}
for all $g\in G$, $x\in G/H$. We note that $\|\tilde P(g)\|_2 \leq \|P\|_\infty$ and $\|\tilde Q(g)\|_2 \leq \|Q\|_\infty$ for every $g\in G$. Using the cocycle identity and the invariance of $\mu$ under the action of $G$, we find that
\begin{align*}
\la \tilde P(g), \tilde Q(h)\ra
&= \frac{1}{\mu(F)} \int 1_{g^{-1}F \cap h^{-1}F}(x) \ \la  P(\alpha(g,x)) , Q(\alpha(h,x)) \ra \ d\mu(x) \\
&= \frac{1}{\mu(F)} \int 1_{g^{-1}F \cap h^{-1}F}(x) \ u(\alpha(g,x) \alpha(h,x)^{-1} ) \ d\mu(x) \\
&= \frac{1}{\mu(F)} \int 1_{g^{-1}F \cap h^{-1}F}(x) \ u(\alpha(gh^{-1},hx)) \ d\mu(x)  \\
&= \frac{1}{\mu(F)} \int 1_{F\cap (gh^{-1})^{-1}F }(x) \ u(\alpha(gh^{-1},x)) \ d\mu(x)  \\
&= v(gh^{-1}).
\end{align*}
Thus, $v\in B_2(G)$ by Proposition~\ref{prop:Gilbert} and $\|v\|_{B_2}\leq \|u\|_{B_2} \leq C$.

To see that $v\in C_0(G)$ we let $\delta > 0$ be given. Since $u\in C_0(H)$ there is a compact set $M\subseteq H$ such that $h\notin M$ implies $|u(h)| \leq \delta$.

If $x\in G/H$ and $g\in G$ is such that $x\in F\cap g^{-1} F$ and $\alpha(g,x) \in M$, then $g\in \sigma(F)M\sigma(F)^{-1}$, which is pre-compact since $\sigma$ is regular. Then it is not hard to see that if $g\notin \sigma(F)M\sigma(F)^{-1}$ then
$$
|v(g)| \leq \frac{1}{\mu(F)} \int_{F\cap g^{-1} F} |u(\alpha(g,x))| \ d\mu(x) \leq \delta.
$$
This proves that $v\in C_0(G)$.

Finally, suppose $g\in L$. We show that $|v(g) - 1| \leq 2\epsilon$. If $x\in F$, then $\alpha(g,x) \in K$ and $|u(\alpha(g,x)) - 1| \leq \epsilon$. Hence
\begin{align*}
|v(g) - 1 |
&= \frac{1}{\mu(F)} \left| \int 1_{F\cap g^{-1}F}(x) u(\alpha(g,x)) - 1_{F\cap g^{-1}F}(x)   - 1_{F\setminus g^{-1}F}(x) \ d\mu(x) \right|\\
&\leq  \frac{1}{\mu(F)} \int 1_{F\cap g^{-1}F}(x) | u(\alpha(g,x)) - 1 |  + 1_{F\setminus g^{-1}F}(x) \ d\mu(x) \\
&\leq \epsilon + \frac{\mu(F\setminus g^{-1}F)}{\mu(F)} \\
&\leq 2\epsilon.
\end{align*}
\end{proof}

\begin{cor}
Let $N$ be a closed normal subgroup in a locally compact group $G$. Assume that $G$ is either second countable or discrete. If $N$ has the weak Haagerup property and $G/N$ is amenable, then $G$ has the weak Haagerup property. In fact, $\Lambda_\WH(G) = \Lambda_\WH(N)$.
\end{cor}

\section{The weak Haagerup property for simple Lie groups}

This section contains results from \cite{HK-WH-examples} about the weak Haagerup property for connected simple Lie groups. The results are merely included here for completeness. The results are consequences of some of the hereditary properties proved here in Section~\ref{sec:hereditary-1} combined with work of de~Laat and Haagerup \cite{MR3047470}, \cite{delaat2}. But before we mention the results, we summarize the situation concerning connected simple Lie groups, the Haagerup property and weak amenability.

Since compact groups are amenable, they also posses the Haagerup property, and they are weakly amenable. So only the non-compact case is of interest. It is known which connected simple Lie groups have the Haagerup property (see \cite[p.~12]{MR1852148}). We summarize the result.
\begin{thm}[\cite{{MR1852148}}]
Let $G$ be a non-compact connected simple Lie group. Then $G$ has the Haagerup property if and only if $G$ is locally isomorphic to either $\SO_0(1,n)$ or $\SU(1,n)$. Otherwise, $G$ has property (T).
\end{thm}

Concerning weak amenability the situation is more subtle, if one wants to include the weak amenability constant, but still the full answer is known.
\begin{thm}[\cite{MR748862},\cite{MR996553},\cite{MR784292},\cite{MR1418350},\cite{UH-preprint},\cite{MR1079871}]\label{thm:WA-Lie-rank-one}
Let $G$ be a non-compact connected simple Lie group. Then
$$
\Lambda_{\WA}(G) = \begin{cases}
1 & \text{ for } G\approx \SO(1,n) \\
1 & \text{ for } G\approx\SU(1,n) \\
2n-1 & \text{ for } G\approx\Sp(1,n) \\
21 & \text{ for } G\approx \mr F_{4(-20)}. \\
\infty & \text{ otherwise }. \\
\end{cases}
$$
\end{thm}

Here $\approx$ denotes local isomorphism. We remark that in the above situation $\Lambda_\WA(G) = 1$ in exactly the same cases as where $G$ has the Haagerup property.

If the only concern is whether or not $\Lambda_\WA(G) < \infty$, i.e., whether or not $G$ is weakly amenable, then the result can be rephrased as follows.

\begin{cor}[\cite{MR748862}, \cite{MR996553}, \cite{MR784292}, \cite{UH-preprint},\cite{MR1079871}]
A connected simple Lie group is weakly amenable if and only if it has real rank zero or one.
\end{cor}

As mentioned earlier, $\Lambda_\WH(G)\leq \Lambda_\WA(G)$ for every locally compact group $G$, and there are examples to show that the inequality can be strict in the most extreme sense: $\Lambda_\WA(H) = \infty$ and $\Lambda_\WH(H) = 1$, when $H = \Z/2 \wr \F_2$. For connected simple Lie groups, however, it turns out that the weak Haagerup property behaves like weak amenability. The following is proved in \cite{HK-WH-examples} using results of \cite{MR3047470}, \cite{delaat2}.

\begin{thm}[\cite{HK-WH-examples}]
A connected simple Lie group has the weak Haagerup property if and only if it has real rank zero or one.
\end{thm}

\section{The weak Haagerup property for von Neumann algebras}
In this section we introduce the weak Haagerup property for finite von Neumann algebras, and we prove that a group von Neumann algebra has this property, if and only if the group has the weak Haagerup property.

In the following, let $M$ be a (finite) von Neumann algebra with a faithful normal trace $\tau$. By a trace we always mean a tracial state. We denote the induced inner product on $M$ by $\la \ ,\ \ra_\tau$. In other words, $\la x,y\ra_\tau = \tau(y^*x)$ for $x,y\in M$. The completion of $M$ with respect to this inner product is a Hilbert space, denoted $L^2(M,\tau)$ or simply $L^2(M)$. The norm on $L^2(M)$ is denoted $\|\ \|_2$ or $\|\ \|_\tau$ and satisfies $\|x\|_2 \leq \|x\|$ for every $x\in M$, where $\|\ \|$ denotes the operator norm on $M$.

When $T:M\to M$ is an bounded operator on $M$, it will be relevant to know sufficient conditions for $T$ to extend to a bounded operator on $L^2(M)$. The following result uses a standard interpolation technique.

\begin{prop}\label{prop:extendable}
Let $(M,\tau)$ be a finite von Neumann algebra with faithful normal trace, and let $S:M\to M$ and $T:M\to M$ be bounded operators on $M$. Suppose $\la Tx,y\ra_\tau = \la x,Sy\ra_\tau$ for every $x,y\in M$. Then $T$ extends to a bounded operator $\tilde T$ on $L^2(M)$, and $\|\tilde T\| \leq \max\{ \|T\|,\|S\|\}$.
\end{prop}

\begin{proof}
After scaling both $T$ and $S$ with $\max\{1, \|T\|,\|S\|\}^{-1}$, we may assume that $\|S\|\leq 1$ and $\|T\| \leq 1$. By \cite[Theorem 5]{MR0355624} the set of invertible elements in $M$ is norm dense, since $M$ is finite. Hence it suffices to prove that $\|Tx\|_2 \leq \| x \|_2$ for every invertible $x\in M$. We prove first that $\|Tx\|_1 \leq \|x \|_1$, and an interpolation technique will then give the result.

Let $M_1$ denote the unit ball of $M$. Recall that $\|x\|_1 = \tau(|x|) = \sup\{ |\tau(y^*x)| \mid y\in M_1\}$. Hence
$$
\|Tx\|_1 = \sup_{y\in M_1} |\tau(y^*Tx)| = \sup_{y\in M_1} |\tau((Sy)^*x)| \leq \sup_{z\in M_1} |\tau(z^*x)| = \|x\|_1.
$$
Since also $\|Tx\| \leq \|x\|$, it follows by an interpolation argument that $\|Tx\|_2 \leq \|x\|_2$. The interpolation argument goes as follows.

Assume for simplicity that $\|x\|_2 \leq 1$. We will show that $\|Tx\|_2 \leq 1$. Since $x$ is invertible, it has polar decomposition $x = uh$, where $u$ is unitary, and $h \geq 0$ is invertible. For $s\in\C$ define
$$
F(s) = uh^{2s}, \quad G(s) = T(F(s)), \quad g(s) = \tau(G(s)G(1-\bar s)^*).
$$
Since $h$ is positive and invertible, $F$ is well-defined and analytic. It follows that $G$ and $g$ are analytic as well.

Next we show that $g$ is bounded on the vertical strip $\Omega = \{s\in\C \mid 0 \leq \Re(s) \leq 1 \}$. Since $\tau$ and $T$ are bounded, it suffices to see that $F$ is bounded on $\Omega$. We have
$$
\|F(s)\| = \|uh^{2s}\| \leq \| h^{2\Re(s)} \| \leq \sup_{0\leq t\leq 1} \|h^{2t}\| < \infty.
$$
Observe that if $v$ and $w$ are unitaries in $M$, and $w$ commutes with $y\in M$, then $|vyw| = |y|$, and hence $\|vyw\|_1 = \|y\|_1$. On the boundary of $\Omega$ we have the following estimates.
$$
\|G(it)\| = \| T(uh^{2it})\| \leq 1,
$$
since $\|T\|\leq 1$, and $u$ and $h^{2it}$ are unitaries. Also
$$
\|G(1+it)\|_1 = \| T(uh^2h^{2it}) \|_1 \leq \|uh^2h^{2it}\|_1 = \|h^2\|_1 = \|x\|_2^2 \leq 1,
$$
It follows that
$$
|g(it)| = |\tau(G(it)G(1+it)^*)| \leq \|G(it)\| \ \|G(1+it)\|_1 \leq 1
$$
and
$$
\|g(1+it)\| = |\tau(G(1+it)G(it)^*)| \leq \|G(it)\| \ \|G(1+it)\|_1 \leq 1.
$$
In conclusion, $g$ is an entire function, bounded on the strip $\Omega$ and bounded by $1$ on the boundary of $\Omega$. It follows from the Three Lines Theorem that $|g(s)| \leq 1$ whenever $s\in\Omega$.

Finally, observe that $g(\frac12) = \tau(Tx(Tx)^*) = \|Tx\|_2^2$. This proves $\|Tx\|_2 \leq 1$. Hence $T$ extends to a bounded operator on $L^2(M)$ of norm at most one.
\end{proof}

\begin{defi}\label{defi:WH-VN}
Let $M$ be a von Neumann algebra with a faithful normal trace $\tau$. Then $(M,\tau)$ has the \emph{weak Haagerup property}, if there is a constant $C > 0$ and a net $(T_\alpha)$ of normal, completely bounded maps on $M$ such that
\begin{enumerate}
	\item $\|T_\alpha\|_\cb \leq C$ for every $\alpha$,
	\item $\la T_\alpha x,y\ra_\tau = \la x,T_\alpha y\ra_\tau$ for every $x,y\in M$,
	\item each $T_\alpha$ extends to a \emph{compact} operator on $L^2(M,\tau)$,
	\item $T_\alpha x \to x$ ultraweakly for every $x\in M$.
\end{enumerate}
The weak Haagerup constant $\Lambda_{\WH}(M,\tau)$ is defined as the infimum of those $C$ for which such a net $(T_\alpha)$ exists, and if no such net exists we write $\Lambda_{\WH}(M,\tau) = \infty$. It is not hard to see that the infimum is actually a minimum and that $\Lambda_\WH(M,\tau) \geq 1$. If $\tau$ is implicit from the context (which will always be the case later on), we simply write $\Lambda_\WH(M)$ for $\Lambda_\WH(M,\tau)$.
\end{defi}

\begin{rem}\label{rem:independent-trace}
The weak Haagerup constant of $M$ is actually independent of the choice of faithful normal trace on $M$, that is, $\Lambda_\WH(M,\tau) = \Lambda_\WH(M,\tau')$ for any two faithful, normal traces $\tau$ and $\tau'$ on $M$. This is Proposition~\ref{prop:independent-trace}.
\end{rem}

\begin{rem}
Note that by Proposition~\ref{prop:extendable}, condition (2) ensures that each $T_\alpha$ extends to a bounded operator on $L^2(M,\tau)$, and the extension is a self-adjoint operator on $L^2(M,\tau)$ with norm at most $\|T_\alpha\|$.
\end{rem}

\begin{rem}\label{rem:topology}
The choice of topology in which the net $(T_\alpha)$ converges to the identity map on $M$ could be one of many without affecting the definition, as we will see now.

Suppose we are given a net $(T_\alpha)$ of normal, completely bounded maps on $M$ such that
\begin{enumerate}
	\item[(1)] $\|T_\alpha\|_\cb \leq C$ for every $\alpha$,
	\item[(2)] $\la T_\alpha x,y\ra_\tau = \la x,T_\alpha y\ra_\tau$ for every $x,y\in M$,
	\item[(3)] each $T_\alpha$ extends to a compact operator on $L^2(M,\tau)$,
	\item[(4)] $T_\alpha \to 1_M$ in the point-weak operator topology.
\end{enumerate}

Since the closure of any convex set in $B(M,M)$ in the point-weak operator topology coincides with its closure in the point-strong operator topology, there is a net $(S_\beta)$ such that $S_\beta\in \conv \{ T_\alpha\}_{\alpha}$ and
\begin{enumerate}
	\item[(1')] $\|S_\beta\|_\cb \leq C$ for every $\beta$,
	\item[(2')] $\la S_\beta x,y\ra_\tau = \la x,S_\beta y\ra_\tau$ for every $x,y\in M$,
	\item[(3')] each $S_\beta$ extends to a compact operator on $L^2(M,\tau)$,
	\item[(4')] $S_\beta \to 1_M$ in the point-strong operator topology.
\end{enumerate}

Since the net $(S_\beta)$ is norm-bounded and the strong operator topology coincides with the trace norm topology on bounded sets of $M$, condition (4') is equivalent to
\begin{enumerate}
	\item[(4'')] $\|S_\beta x - x\|_2 \to 0$ for any $x\in M$.
\end{enumerate}
If we let $\tilde S_\beta$ denote the extension of $S_\beta$ to an operator on $L^2(M)$, then by Proposition~\ref{prop:extendable} $\|\tilde S_\beta\| \leq \| S_\beta\|$, so the net $(\tilde S_\beta)$ is bounded, and hence (4'') is equivalent to the condition that
\begin{enumerate}
	\item[(4''')] $\tilde S_\beta \to 1_{L^2(M)}$ strongly.
\end{enumerate}
Using that $\|y^*\|_2 = \|y\|_2$ for any $y\in M$, condition (4'') implies that
\begin{enumerate}
	\item[(4'''')] $\|(S_\beta x)^* - x^*\|_2 \to 0 \text{ for any } x\in M$
\end{enumerate}
so also, $S_\beta\to 1_M$ in the point-strong$^*$ operator topology. Finally, since the net $(S_\beta)$ is bounded in norm, and since the ultrastrong and strong operator topologies coincide on bounded sets, we also obtain
\begin{enumerate}
	\item[(4''''')] $S_\beta \to 1_M$ in the point-ultrastrong$^*$ operator topology.
\end{enumerate}

\end{rem}

Let us see that the weak Haagerup property is indeed weaker than the (usual) Haagerup property. Let $M$ be a von Neumann algebra with a faithful normal trace $\tau$. We recall (see \cite{MR1372231},\cite{MR1962471}) that $(M,\tau)$ has the Haagerup property if there exists a net $(T_\alpha)_{\alpha\in A}$ of normal completely positive maps from $M$ to itself such that
\begin{enumerate}
	\item $\tau\circ T_\alpha \leq \tau$ for every $\alpha$,
	\item $T_\alpha$ extends to a compact operator on $L^2(M)$,
	\item $\|T_\alpha x - x \|_2 \to 0$ for every $x\in M$.
\end{enumerate}
One can actually assume that $\tau\circ T_\alpha = \tau$ and that $T_\alpha$ is unital (see \cite[Proposition~2.2]{MR1962471}). Moreover, the Haagerup property does not depend on the choice of $\tau$ (see \cite[Proposition~2.4]{MR1962471}).

\begin{prop}
Let $M$ be a von Neumann algebra with a faithful normal trace $\tau$. If $(M,\tau)$ has the Haagerup property, then $(M,\tau)$ has the weak Haagerup property. In fact, $\Lambda_\WH(M,\tau) = 1$.
\end{prop}
\begin{proof}
The proof is merely an application of the following result (see \cite[Lemma~2.5]{MR2240699}). If $T$ is a normal unital completely positive map on $M$, then $\tau\circ T = \tau$ if and only if there is a normal unital completely positive map $S\colon M\to M$ such that $\la Tx,y\ra_\tau = \la x,Sy\ra_\tau$ for every $x,y\in M$.

Suppose $M$ has the Haagerup property and let $(T_\alpha)_{\alpha\in A}$ be a net of normal unital completely positive maps from $M$ to itself such that
\begin{itemize}
	\item $\tau\circ T_\alpha = \tau$ for every $\alpha$,
	\item $T_\alpha$ extends to a compact operator on $L^2(M)$,
	\item $\|T_\alpha x - x \|_2 \to 0$ for every $x\in M$.
\end{itemize}
Then there are normal unital completely positive maps $S_\alpha\colon M\to M$ such that $\la T_\alpha x,y\ra_\tau = \la x,S_\alpha y\ra_\tau$ for every $x,y\in M$. Let $R_\alpha = \frac12(T_\alpha + S_\alpha)$. Then $R_\alpha$ is normal unital completely positive and 
\begin{itemize}
	\item $\la R_\alpha x,y\ra_\tau = \la x, R_\alpha y\ra_\tau$ for every $\alpha$,
	\item $R_\alpha$ extends to a compact operator on $L^2(M)$,
	\item $\|R_\alpha x - x \|_2 \to 0$ for every $x\in M$.
\end{itemize}
Since unital completely positive maps have completely bounded norm 1, this shows that $\Lambda_\WH(M,\tau) \leq 1$. This completes the proof.
\end{proof}

It is mentioned in \cite{MR1962471} that injective finite von Neumann algebras have the Haagerup property. Indeed, it is a deep, and by now classical, result that injective von Neumann algebras are semidiscrete \cite{MR0405117}, \cite{MR0430794}, \cite{MR0454659} (see \cite[Theorem~9.3.4]{MR2391387} for a proof of the finite case based on \cite{MR0448108}). It then follows from \cite[Proposition~4.6]{MR2279097} that injective von Neumann algebras which admit a faithful normal trace have the Haagerup property. In particular, injective von Neumann algebras with a faithful normal trace have the weak Haagerup property.

We now turn to discrete groups and their group von Neumann algebras. For the moment, fix a discrete group $\Gamma$. We let $\lambda$ denote the left regular representation of $\Gamma$ on $\ell^2(\Gamma)$. The von Neumann algebra generated by $\lambda(\Gamma)$ inside $B(\ell^2(\Gamma))$ is the \emph{group von Neumann algebra} denoted $L(\Gamma)$. It is equipped with the faithful normal trace $\tau$ given by $\tau(x) = \la x\delta_e,\delta_e\ra$ for $x\in L(\Gamma)$.

\begin{repthm}{thm:vna}
Let $\Gamma$ be a discrete group. The following conditions are equivalent.
\begin{enumerate}
	\item The group $\Gamma$ has the weak Haagerup property.
	\item The group von Neumann algebra $L(\Gamma)$ (equipped with its canonical trace) has the weak Haagerup property.
\end{enumerate}
More precisely, $\Lambda_\WH(\Gamma) = \Lambda_\WH(L(\Gamma))$.
\end{repthm}
\begin{proof}
Suppose the net $( u_\alpha)$ of maps in $B_2(\Gamma)\cap C_0(\Gamma)$ witnesses the weak Haagerup property of $\Gamma$ with $\| u_\alpha\|_{B_2} \leq C$ for every $\alpha$. Upon replacing $ u_\alpha$ with $\frac12( u_\alpha + \bar u_\alpha)$ we may assume that $ u_\alpha$ is real. Let $T_\alpha = M_{u_\alpha}$ be the corresponding multiplier on $L(\Gamma)$, that is
\begin{align}\label{eq:diagonal}
T_\alpha \lambda(g) =  u_\alpha(g)\lambda(g), \quad g\in\Gamma.
\end{align}
Then $T_\alpha$ is normal and completely bounded on $L(\Gamma)$ with $\|T_\alpha\|_\cb = \| u_\alpha\|_{B_2}$. From \eqref{eq:diagonal} it follows that $T_\alpha$ extends to a diagonal operator $\tilde T_\alpha$ on $L^2(L(\Gamma))$, when $L^2(L(\Gamma))$ has the standard basis $\{\lambda(g)\}_{g\in G}$. Since $ u_\alpha$ is real, $\tilde T_\alpha$ is self-adjoint. In particular $\la T_\alpha x,y \ra_\tau = \la x,T_\alpha y\ra_\tau$ for all $x,y\in L(\Gamma)$. Also, $\tilde T_\alpha$ is compact, because $ u_\alpha \in C_0(\Gamma)$. Since $ u_\alpha \to 1$ pointwise and $\| u_\alpha\|_\infty \leq C$, it follows that $\tilde T_\alpha \to 1_{L^2}$ strongly on $L^2(L(\Gamma))$. By Remark~\ref{rem:topology}, this proves that $L(\Gamma)$ has the weak Haagerup property with $\Lambda_\WH(L(\Gamma)) \leq C$.

Conversely, suppose there is a net $(T_\alpha)$ of maps on $L(\Gamma)$ witnessing the weak Haagerup property of $L(\Gamma)$ with $\|T_\alpha\|_\cb \leq C$ for every $\alpha$. Let
$$
 u_\alpha(g) = \tau(\lambda(g)^*T_\alpha(\lambda(g))).
$$ 

Since $T_\alpha \to \id_{L(\Gamma)}$ point-ultraweakly, and $\tau$ is normal, it follows that $ u_\alpha\to 1$ pointwise.

Let $V:\ell^2(\Gamma)\to\ell^2(\Gamma)\tensor\ell^2(\Gamma)$ be the isometry given by $V\delta_g = \delta_g \tensor \delta_g$. Observe then that
$$
V^*(\lambda(g)\tensor\lambda(h))V = \begin{cases}
\lambda(g) & \text{if } g = h, \\
0 & \text{if } g \neq h, \\
\end{cases}
$$
so
$$
V^*(\lambda(g)\tensor a)V = \tau(\lambda(g)^*a) \lambda(g).
$$
By Fell's absorption principle \cite[Theorem 2.5.5]{MR2391387} there is a normal $^*$-homomorphism $\sigma:L(\Gamma)\to L(\Gamma)\tensor L(\Gamma)$ such that $\sigma(\lambda(g))= \lambda(g)\tensor\lambda(g)$. Using Lemma~\ref{lem:tensor-cc} we see that the operator $\id_{L(\Gamma)} \tensor T_\alpha$ on $L(\Gamma)\tensor L(\Gamma)$ exists, and it is easily verified that
$$
V^*((\id_{L(\Gamma)} \tensor T_\alpha)(\lambda(g)\tensor\lambda(g)))V =  u_\alpha(g)\lambda(g),
$$
when $g\in\Gamma$, and so
$$
V^*((\id_{L(\Gamma)} \tensor T_\alpha)(\sigma(a)))V = M_{ u_\alpha}(a) \quad \text{for all } a\in L(\Gamma).
$$
It follows that $M_{ u_\alpha}$ is completely bounded and $ u_\alpha\in B_2(\Gamma)$ with 
$$
\| u_\alpha\|_{B_2} = \|M_{ u_\alpha}\|_\cb \leq \|T_\alpha\|_\cb \leq C,
$$
where the first inequality follows from Proposition~D.6 in \cite{MR2391387}.

It remains to show that $ u_\alpha\in C_0(\Gamma)$. We may of course suppose that $\Gamma$ is infinite. Since $T_\alpha$ extends to a compact operator on $L^2(L(\Gamma))$, it follows that
$$
\lim_g \|T_\alpha(\lambda(g)) \|_2 = 0,
$$
because $(\lambda(g))_{g\in\Gamma}$ is orthonormal in $L^2(\Gamma)$. By the Cauchy-Schwarz inequality
$$
| u_\alpha(g)| \leq \|T_\alpha\lambda(g)\|_2 \to 0 \qquad \text{as } g\to \infty.
$$
This completes the proof.
\end{proof}

\section{Hereditary properties II}\label{sec:hereditary2}
In this section we prove hereditary results for the weak Haagerup property of von Neumann algebras.
As an application we are able to show that the weak Haagerup property of a von Neumann algebra does not depend on the choice of the faithful normal trace.

When $M$ is a finite von Neumann algebra with a faithful normal trace $\tau$, and $p\in M$ is a non-zero projection, we let $\tau_p$ denote the faithful normal trace on $pMp$ given as $\tau_p(x) = \tau(p)^{-1} \tau(x)$.

Since we have not yet proved that the weak Haagerup property of a von Neumann algebra does not depend on the choice of faithful normal trace (Proposition~\ref{prop:independent-trace}), we state Theorem~\ref{thm:hereditary2} in the following more cumbersome way. Once we have shown Proposition~\ref{prop:independent-trace}, Theorem~\ref{thm:hereditary2} makes sense and is Theorem~\ref{thm:hereditary2-2}

\begin{customthm}{\ref*{thm:hereditary2}'}\label{thm:hereditary2-2}
Let $(M,\tau)$, $(M_1,\tau_1)$ and $(M_2,\tau_2)$ be a finite von Neumann algebras with a faithful normal traces.

\begin{enumerate}
	\item Suppose $(M,\tau)$ is weakly Haagerup with constant $C$, and $N\subseteq M$ is a von Neumann subalgebra. Then $(N,\tau)$ is weakly Haagerup with constant at most $C$.
	\item Suppose $(M,\tau)$ is weakly Haagerup with constant $C$, $p\in M$ is a non-zero projection. Then $(pMp,\tau_p)$ is weakly Haagerup with constant at most $C$.
	\item Suppose $1\in N_1\subseteq N_2\subseteq \cdots$ are von Neumann subalgebras of $M$ generating all of $M$, and that there is an increasing sequence of non-zero projections $p_n\in N_n$ with strong limit $1$. If each $(p_nN_np_n,\tau_{p_n})$ is weakly Haagerup with constant at most $C$, then $(M,\tau)$ is weakly Haagerup with constant at most $C$.
	\item Suppose $(M_1,\tau_1)$, $(M_2,\tau_2)$, $\ldots$ is a (possibly finite) sequence of von Neumann algebras with faithful normal traces, and that $\alpha_1,\alpha_2,\ldots$ are strictly positive numbers with $\sum_n \alpha_n = 1$. Then the weak Haagerup constant of
	$$
	\left(\bigoplus_n M_n, \bigoplus_n \alpha_n\tau_n\right)
	$$
	equals $\sup_n \Lambda_\WH(M_n,\tau_n)$, where $\bigoplus_n \alpha_n\tau_n$ denotes the trace defined by
	$$
	\left(\bigoplus_n \alpha_n\tau_n\right)(x_n) = \sum_n \alpha_n\tau_n(x_n), \quad (x_n)_n \in \bigoplus_n M_n.
	$$.
	\item Suppose $(M_1,\tau_1)$ and $(M_2,\tau_2)$ are weakly Haagerup with constant $C_1$ and $C_2$, respectively. Then the tensor product $(M_1\vntensor M_2,\tau_1\vntensor\tau_2)$ is weakly Haagerup with constant at most $C_1C_2$.
\end{enumerate}
\end{customthm}

\begin{proof}\mbox{}

(1) Let $E:M\to N$ be the unique trace-preserving conditional expectation. Given a net $(T_\alpha)$ witnessing the weak Haagerup property of $M$ we let $S_\alpha = E\circ T_\alpha|_N$. Clearly, $\|S_\alpha\|_\cb \leq \|T_\alpha\|_\cb$. Since $E$ is an $N$-bimodule map, trace-preserving and positive, an easy calculation shows that $\la S_\alpha x,y\ra = \la x,S_\alpha y\ra$ for every $x,y\in N$.

As is customary, the Hilbert space $L^2(N)$ is naturally identified with the closed subspace of $L^2(M)$ spanned by $N \subseteq  M \subseteq L^2(M)$, and the conditional expectation $E\colon M\to N$ extends to a projection $e_N:L^2(M)\to L^2(N)$. Since $T_\alpha$ extends to a compact operator $\tilde T_\alpha$ on $L^2(M)$, it follows that $E\circ T_\alpha$ extends to the compact operator $e_N\tilde T_\alpha$ on $L^2(M)$. Hence $S_\alpha$ extends to the compact operator $e_N\tilde T_\alpha|_{L^2(N)}$ on $L^2(N)$.

Since $E$ is normal, $E|_N = 1_N$, and $T_\alpha \to 1_M$ point-ultraweakly, we obtain $S_\alpha\to 1_N$ point-ultraweakly.

(2) Let $P:M\to pMp$ be the map $P(x) = pxp$, $x\in M$. Then $P$ is unital and completely positive. Given a net $(T_\alpha)$ witnessing the weak Haagerup property of $M$ we let $S_\alpha = P\circ T_\alpha|_{pMp}$. Clearly, $\|S_\alpha\|_\cb \leq \|T_\alpha|_{pMp}\|_\cb \leq \|T_\alpha\|_\cb$. An easy calculation shows that
$$
\la S_\alpha x,y\ra_{\tau_p} = \la x,S_\alpha y\ra_{\tau_p} \quad\text{for all } x,y\in pMp.
$$

Let $V:L^2(pMp)\to L^2(M)$ be the map $Vx = \tau(p)^{-1/2}x$. Then $V$ is an isometry, and evidently $V^*x = \tau(p)^{1/2} pxp$ for every $x\in M$. It follows that on $pMp$ we have $S_\alpha = V^*T_\alpha V$. Hence $S_\alpha$ extends to the compact operator
$$
\tilde S_\alpha = V^*\tilde T_\alpha V,
$$
on $L^2(pMp)$, where $\tilde T_\alpha$ denotes the extension of $T_\alpha$ to a compact operator on $L^2(M)$.

Since $P$ is normal, it follows that $S_\alpha\to 1$ point-ultraweakly.

(3) We denote the trace-preserving conditional expectation $M\to N_n$ by $E_n$ and its extension to a projection $L^2(M)\to L^2(N_n)$ by $e_n$. Note first that since $M$ is generated by the sequence $N_n$, for each $x\in M$ we have $E_n(x) \to x$ strongly. Indeed, the union of the increasing sequence of Hilbert spaces $L^2(N_n)$ is a norm dense subspace of the Hilbert space $L^2(M)$, and thus $e_n\nearrow 1_{L^2(M)}$ strongly. In other words, $\|E_n(x) - x \|_\tau \to 0$.

For each $n\in\N$ we let $S_n : M\to p_nN_np_n$ be $S_n(x) = p_nE_n(x)p_n$. It follows that $S_n(x) \to x$ strongly.

Let $F\subseteq M$ be a finite set, and let $\epsilon > 0$ be given. Choose $n$ such that
$$
\|S_n(x) - x\|_\tau \leq \epsilon \quad\text{for all } x\in F.
$$
By assumption there is a completely bounded map $R:p_nN_np_n\to p_nN_np_n$ such that $\|R\|_\cb \leq C$, $R$ extends to a self-adjoint compact operator on $L^2(p_nN_np_n)$, and
$$
\|R(S_n(x)) - S_n(x)\|_{\tau_{p_n}} \leq \epsilon \quad\text{for all } x\in F.
$$
Let $T_\alpha = R\circ S_n$, where $\alpha = (F,\epsilon)$. Clearly,
$$
\|T_\alpha x - x \|_\tau \leq 2\epsilon \quad\text{when } x\in F.
$$
It follows that $T_\alpha \to 1$ point-strongly.

Since $S_n$ is unital and completely positive, we get $\|T_\alpha\|_\cb \leq \|R\|_\cb \leq C$. When $x,y\in M$ we have
\begin{align*}
\la T_\alpha x,y\ra_\tau &= \la R(p_nE_n(x)p_n), p_nE_n(y)p_n \ra_\tau \\
&= \la p_nE_n(x)p_n, R(p_nE_n(y)p_n)\ra_\tau = \la x,T_\alpha y\ra_\tau
\end{align*}
using the properties of $E_n$ and $R$.
Since $T_\alpha$ is the composition
$$
\xymatrix{
M \ar[r]^{E_n} & N_n \ar[r]^-{P_n} & p_nN_np_n \ar[r]^{R} & p_nN_np_n \ar[r]^-{\iota} & N_n \ar[r]^{\iota} & M,
}
$$
where $\iota$ denotes inclusion, it follows that the extension of $T_\alpha$ to $L^2(M)$ is compact, because the extension of $R$ to $L^2(p_nN_np_n)$ is compact:
$$
\xymatrix@-0.5pc{
L^2(M) \ar[r]^{e_n} & L^2(N_n) \ar[r]^-{\tilde P_n} & L^2(p_nN_np_n) \ar[r]^{\tilde R} & L^2(p_nN_np_n) 
\ar[r]^-{\tilde\iota} & L^2(N_n) \ar[r]^{\tilde \iota} & L^2(M).
}
$$
The net $(T_\alpha)_{\alpha\in A}$ indexed by $A = \{(F,\epsilon) \mid F\subseteq M \text{ finite},\ \epsilon > 0\}$ shows that the weak Haagerup constant of $M$ is at most $C$.

(4) It is enough to show that the weak Haagerup constant of $M_1 \oplus M_2$ with respect to the trace $\tau = \lambda \tau_1 \oplus (1-\lambda)\tau_2$
	equals
	$$
	\max\{ \Lambda_\WH(M_1,\tau_1), \Lambda_\WH(M_2,\tau_2) \}
	$$
	for any $0 < \lambda < 1$, and then apply induction and (3) to obtain the general case of (4). We only prove
	\begin{align}\label{eq:direct-sum}
	\Lambda_\WH(M_1 \oplus M_2) \leq \max\{ \Lambda_\WH(M_1), \Lambda_\WH(M_2) \},
	\end{align}
	since the other inequality is clear from (2).
	
	Two points should be made. Firstly, if $T_1$ and $T_2$ are normal completely bounded maps on $M_1$ and $M_2$ respectively, then $T_1 \oplus T_2$ is a normal completely bounded map on $M$ with completely bounded norm
	$$
	\|T_1 \oplus T_2 \|_\cb = \max\{ \|T_1\|_\cb, \|T_2\|_\cb \}.
	$$
	Secondly, the map $V(x\oplus y) = \lambda^{1/2} x \oplus (1-\lambda)^{1/2} y$ on $M_1\oplus M_2$ extends to a unitary operator
	$$
	V\colon L^2(M_1\oplus M_2,\tau) \to L^2(M_1,\tau_1) \oplus L^2(M_2,\tau_2).
	$$
	Now, let $\epsilon > 0$ be given and let $(S_\alpha)_{\alpha\in A}$ and $(T_\beta)_{\beta\in B}$ be normal completely bounded maps on $M_1$ and $M_2$, respectively such that
\begin{itemize}
	\item $\|S_\alpha\|_\cb \leq \Lambda_\WH(M_1,\tau_1) + \epsilon$ for every $\alpha$,
	\item $\la S_\alpha x,y\ra_{\tau_1} = \la x,S_\alpha y\ra_{\tau_1}$ for every $x,y\in M_1$,
	\item each $S_\alpha$ extends to a \emph{compact} operator on $L^2(M_1,\tau_1)$,
	\item $S_\alpha x \to x$ ultraweakly for every $x\in M_1$,
\end{itemize}
and similar properties holds for $(T_\beta)_{\beta\in B}$ and $M_2$. We may assume that $A = B$. Now, let $R_\alpha = S_\alpha \oplus T_\alpha$. Using the net $(R_\alpha)$ it is easy to show that
$$
\Lambda_\WH(M_1 \oplus M_2) \leq \max\{ \Lambda_\WH(M_1), \Lambda_\WH(M_2) \} + \epsilon.
$$
Letting $\epsilon \to 0$ we obtain \eqref{eq:direct-sum}.

(5) We remark that the product trace $\tau_1\vntensor\tau_2$ on the von Neumann algebraic tensor product $M_1 \vntensor M_2$ is a faithful normal trace (see \cite[Corollary~IV.5.12]{MR1873025}). Suppose we are given nets $(S_\alpha)_{\alpha\in A}$ and $(T_\beta)_{\beta\in B}$ witnessing the weak Haagerup property of $M_1$ and $M_2$, respectively. By Remark~\ref{rem:topology} we may assume that
\begin{align}\label{eq:SO-conv}
\tilde S_\alpha \to 1_{L^2(M_1)} \text{ strongly} \quad\text{and}\quad
\tilde T_\beta \to 1_{L^2(M_2)} \text{ strongly},
\end{align}
where $\tilde S_\alpha$ and $\tilde T_\beta$ denote the extensions to operators on $L^2(M_1)$ and $L^2(M_2)$, respectively. For each $\gamma = (\alpha,\beta) \in A\times B$, we consider the map $R_\gamma = S_\alpha \vntensor T_\beta$ given by Lemma~\ref{lem:tensor-cc} below. Then $R_\gamma$ is a normal, completely bounded map on $M\vntensor N$ with $\|R_\gamma\|_\cb \leq \|S_\alpha\|_\cb \|T_\beta\|_\cb$. Let $\tau = \tau_1\vntensor\tau_2$ be the product trace. We claim that when $A\times B$ is given the product order, the net $(R_\gamma)_{\gamma\in A\times B}$ witnesses the weak Haagerup property of $M_1\vntensor M_2$, i.e. that
\begin{enumerate}
	\item[(a)] $\la R_\gamma x,y\ra_\tau = \la x,R_\gamma y\ra_\tau$ for every $x,y\in M_1 \vntensor M_2$.
	\item[(b)] Each $R_\gamma$ extends to a \emph{compact} operator $\tilde R_\gamma$ on $L^2(M_1\vntensor M_2,\tau)$.
	\item[(c)] $\tilde R_\gamma \to 1_{L^2(M_1\vntensor M_2)}$ strongly.
\end{enumerate}
Condition (a) is easy to check on elementary tensors, and then when $x$ and $y$ are in the algebraic tensor product $M_1\tensor M_2$. Since the unit ball of the algebraic tensor product $M_1\tensor M_2$ is dense in the unit ball of $M_1\vntensor M_2$ in the strong$^*$ operator topology, it follows that (a) holds for arbitrary $x,y\in M_1\vntensor M_2$.

If $V : L^2(M_1) \tensor L^2(M_2) \to L^2(M_1 \vntensor M_2)$ is the unitary which is the identity on $M_1\tensor M_2$, then
$$
R_\gamma = V(S_\alpha \vntensor T_\beta)V^*
$$
Thus, since the tensor product of two compact operators is compact, $R_\gamma$ extends to a compact operator on $L^2(M_1 \vntensor M_2)$.

Condition (c) follows easily from \eqref{eq:SO-conv} and the general fact that if two bounded nets $(V_\alpha)$ and $(W_\beta)$ of operators on Hilbert spaces converge strongly with limits $V$ and $W$, then the net $V_\alpha\tensor W_\beta$ converges strongly to $V\tensor W$.
\end{proof}

In the course of proving (5) above, we postponed the proof of Lemma~\ref{lem:tensor-cc} concerning the existence of the tensor product of two normal, completely bounded map between von Neumann algebras. A version of the lemma exists for completely contractive maps between operator spaces, when the tensor product under consideration is the operator space projective tensor product (see \cite[Proposition~7.1.3]{MR1793753}) or the operator space injective tensor product (see \cite[Proposition~8.1.5]{MR1793753}). The operator space injective tensor product coincides with the minimal \Cs-algebraic tensor product, when the operator spaces are von Neumann algebras (see \cite[Proposition~8.1.6]{MR1793753}). Also, a version of the lemma exists for normal, completely positive maps between von Neumann algebras (\cite[Proposition~IV.5.13]{MR1873025}). See also \cite[Lemma~1.5]{MR784292}.

\begin{lem}\label{lem:tensor-cc}
Suppose $M_i$ and $N_i$ ($i = 1,2$) are von Neumann algebras and $T_i:M_i\to N_i$ are normal, completely contractive maps. Then there is a normal, completely contractive map $T_1 \vntensor T_2 : M_1\vntensor M_2\to N_1\vntensor N_2$ such that
$$
T_1 \vntensor T_2 (x_1\tensor x_2) = T_1x_1 \tensor T_2x_2 \quad\text{ for all } x_i\in M_i \ (i = 1,2).
$$
\end{lem}
\begin{proof}
It follows from \cite[Proposition~8.1.5]{MR1793753} and \cite[Proposition~8.1.6]{MR1793753} that there is a completely contractive map $T_1 \tensor T_2 : M_1\tensor_\mr{min} M_2\to N_1\tensor_\mr{min} N_2$ between the minimal tensor products such that
$$
T_1 \tensor T_2 (x_1\tensor x_2) = T_1x_1 \tensor T_2x_2 \quad\text{ for all } x_i\in M_i \ (i = 1,2).
$$
We must show that $T_1\tensor T_2$ extends continuously to a completely contraction from the ultraweak closure $M_1\vntensor M_2$ of $M_1\tensor_\mr{min} M_2$. First we show that $T_1\tensor T_2$ is ultraweakly continuous. For this, it will suffice to show that $\rho \circ T_1\tensor T_2$ is ultraweakly continuous on $M_1 \tensor_\mr{min} M_2$ for each ultraweakly continuous functional $\rho \in (N_1\vntensor N_2)_*$.

Suppose first that $\rho$ is of the form $\rho_1\tensor\rho_2$ for some $\rho_1\in (N_1)_*$ and $\rho_2\in (N_2)_*$. Then if we let $\sigma_i = \rho_i \circ T_i$, it is clear that $\sigma_1\tensor \sigma_2$ is ultraweakly continuous \cite[11.2.7]{MR1468230}, and $\rho \circ (T_1\tensor T_2) = \sigma_1 \tensor \sigma_2$. In general, $\rho$ is the norm limit of a sequence of functionals $\rho_n$ where each $\rho_n$ is a finite linear combination of ultraweakly continuous product functionals \cite[11.2.8]{MR1468230}, and it then follows from \cite[10.1.15]{MR1468230} that $\rho\circ T_1\tensor T_2$ is ultraweakly continuous.

Now, from \cite[10.1.10]{MR1468230} it follows that $T_1\tensor T_2$ extends (uniquely) to an ultraweakly continuous contraction $M_1\vntensor M_2\to N_1\vntensor N_2$. The same argument applied to $T_1\tensor T_2 \tensor \id_n$, where $\id_n\colon M_n(\C)\to M_n(\C)$ is the identity, shows that $T_1\vntensor T_2$ is not only contractive, but completely contractive.
\end{proof}

\begin{rem}
Theorem~\ref{thm:hereditary2} (1)--(3) may conveniently be expressed as the following inequalities.
\begin{align*}
\Lambda_\WH(N) &\leq \Lambda_\WH(M), \\
\Lambda_\WH(pMp) &\leq \Lambda_\WH(M), \\
\Lambda_\WH(M) &= \sup_{n\in\N} \Lambda_\WH(p_n N_n p_n), 
\end{align*}
when $N\subseteq M$ is a subalgebra, $p\in M$ is a non-zero projection, $(N_n)_{n\geq 1}$ is an increasing sequence of subalgebras generating $M$ with projections $p_n \in N_n$, $p_n\nearrow 1$. Theorem~\ref{thm:hereditary2} (5) reads
\begin{align}\label{eq:WH-multiplication}
\Lambda_{\WH}(M_1\bar\tensor M_2) \leq \Lambda_{\WH}(M_1)\Lambda_{\WH}(M_2).
\end{align}
\end{rem}

\begin{rem}
We do not know if $\Lambda_{\WH}(M_1\bar\tensor M_2) = \Lambda_{\WH}(M_1)\Lambda_{\WH}(M_2)$ holds for any two finite von Neumann algebras $M_1$ and $M_2$. The corresponding result for the weak amenability constant $\Lambda_\WA$ is known to be true, \cite[Theorem~4.1]{MR1323684}. If either $\Lambda_\WH(M_1) = 1$ or $\Lambda_\WH(M_2) = 1$, then equality holds in \eqref{eq:WH-multiplication}.
\end{rem}

We will now show that the weak Haagerup property does not depend on the choice of the faithful normal trace. The basic idea of the proof is to apply the noncommutative Radon-Nikodym theorem. Since the Radon-Nikodym derivative in general may be an unbounded operator, we will need to cut it into pieces that are bounded and then apply Theorem~\ref{thm:hereditary2} (4) in the end.

Let $M$ be a von Neumann algebra acting on a Hilbert space $\mc H$, and let $M'$ denote the commutant of $M$. A (possibly unbounded) closed operator $h$ is affiliated with $M$ if $h u = uh$ (with agreement of domains) for every unitary $u\in M'$. If $h$ is bounded, then by the bicommutant theorem $h$ is affiliated with $M$ if and only if $h\in M$. In general, if $h$ is affiliated with $M$, then $f(h)$ lies in $M$ for every bounded Borel function $f$ on $[0,\infty[$. See e.g. \cite[Appendix~B]{MR2433341} for details.

We recall the version of the Radon-Nikodym theorem that we will need. We refer to \cite{MR0412827} for more details. We denote the center of $M$ by $Z(M)$. Let $\tau$ be a faithful normal trace on $M$ and suppose $h$ is a self-adjoint, positive operator affiliated with $Z(M)$. For $\epsilon > 0$ put $h_\epsilon = h(1+\epsilon h)^{-1}$. Then $h_\epsilon \in Z(M)_+$ for every $\epsilon > 0$. When $x\in M_+$, define the number $\tau(hx)$ by
\begin{align}\label{eq:derived-weight}
\tau(hx) = \lim_{\epsilon\to 0} \tau(h_\epsilon x).
\end{align}
Then $\tau'$ defined by $\tau'(x) = \tau(hx)$ is a normal semifinite weight on $M$. If moreover $\lim_\epsilon \tau(h_\epsilon) = 1$, then \eqref{eq:derived-weight} makes sense for all $x\in M$ and defines a normal trace $\tau'$ on $M$. The Radon-Nikodym theorem \cite[Theorem~5.4]{MR0412827} gives a converse to this: Given any normal trace $\tau'$ on $M$ there is a unique self-adjoint positive operator $h$ affiliated with $Z(M)$ such that $\tau'(x) = \tau(hx)$ for every $x\in M$.

\begin{prop}\label{prop:independent-trace}
Let $M$ be a von Neumann algebra with two faithful normal traces $\tau$ and $\tau'$. Then $M$ has the weak Haagerup property with respect to $\tau$ if and only if $M$ has the weak Haagerup property with respect to $\tau'$. More precisely,
$$
\Lambda_\WH(M,\tau) = \Lambda_\WH(M,\tau').
$$
\end{prop}
\begin{proof}

Let $\epsilon > 0$ be arbitrary. We will show that
\begin{align}\label{eq:trace-proof}
\Lambda_\WH(M,\tau') \leq \Lambda_\WH(M,\tau) (1+\epsilon).
\end{align}
By symmetry and letting $\epsilon\to 0$, this will complete the proof. We may of course assume that $ \Lambda_\WH(M,\tau) < \infty$, since otherwise \eqref{eq:trace-proof} obviously holds.

We let $Z(M)$ denote the center of $M$. Suppose first that there is a positive, invertible operator $h\in Z(M)$ such that $\tau'(x) = \tau(hx)$ for every $x\in M$ and $h$ has spectrum $\sigma(h)$ contained in the interval $[c(1+\epsilon)^n,c(1+\epsilon)^{n+1}]$ for some $c > 0$ and some integer $n$. Note that then
$$
\|h^{1/2}\| \|h^{-1/2}\| \leq (1+\epsilon)^{1/2}.
$$

Let $(T_\alpha)$ be a net of normal, completely bounded operators on $M$ such that

\begin{enumerate}
	\item $\|T_\alpha\|_\cb \leq \Lambda_\WH(M,\tau) (1+\epsilon)^{1/2}$ for every $\alpha$,
	\item $\la T_\alpha x,y\ra_\tau = \la x,T_\alpha y\ra_\tau$ for every $x,y\in M$,
	\item each $T_\alpha$ extends to a compact operator on $L^2(M,\tau)$,
	\item $T_\alpha x \to x$ ultraweakly for every $x\in M$.
\end{enumerate}

Since $h$ belongs to $Z(M)_+$, it is easily verified that the map $U\colon M\to M$ defined by $Ux = h^{1/2}x$ extends to an isometry $L^2(M,\tau') \to L^2(M,\tau)$, and since $h$ is invertible, $U$ is actually a unitary. We let $S_\alpha$ be the operator on $M$ defined as $S_\alpha = U^* T_\alpha U$, that is $S_\alpha x = h^{-1/2} T_\alpha(h^{1/2}x)$. Then $S_\alpha$ is normal and completely bounded with
$$
\|S_\alpha\|_\cb \leq \|h^{1/2}\| \|h^{-1/2}\| \|T_\alpha\|_\cb \leq \Lambda_\WH(M,\tau) (1+\epsilon).
$$
Since $U$ is a unitary, it is clear from (2), (3) and (4) that $S_\alpha$ extends to a self-adjoint, compact operator on $L^2(M,\tau')$ and that $S_\alpha x \to x$ ultraweakly for every $x\in M$. This shows that
$$
\Lambda_\WH(M,\tau') \leq \Lambda_\WH(M,\tau) (1+\epsilon).
$$

In general, there is a (possibly unbounded) unique self-adjoint positive operator $h$ affiliated with $Z(M)$ such that $\tau'(x) = \tau(hx)$.  For each $n\in\Z$ let $p_n$ denote the spectral projection of $h$ defined as
$$
p_n = 1_{[(1+\epsilon)^n,(1+\epsilon)^{n+1}[}(h),
$$
and let $q = 1_{\{0\}}(h)$. Then $p_n$ and $q$ are projections in $Z(M)$. Since (the closure of) $hq$ is zero we see that
$$
\tau'(q) = \tau(hq) = \tau(0) = 0,
$$
and then we must have $q = 0$, since $\tau'$ is faithful. Hence
$$
\sum_{n=-\infty}^\infty p_n = 1_{]0,\infty[}(h) = 1.
$$
Let $I$ be the set of those $n\in\Z$ for which $p_n \neq 0$, and for $n\in I$ let $M_n$ denote the von Neumann algebra $p_n M$ with faithful normal trace $\tau_n = \tau(p_n)^{-1}\tau$. Then from the decomposition
$$
M = \bigoplus_{n\in I} M_n
$$
we get by Theorem~\ref{thm:hereditary2} (4) that
$$
\Lambda_\WH(M,\tau) = \sup_{n\in I} \Lambda_\WH(M_n,\tau_n).
$$
Similarly,
$$
\Lambda_\WH(M,\tau') = \sup_{n\in I} \Lambda_\WH(M_n,\tau'_n),
$$
where $\tau'_n = \tau'(p_n)^{-1}\tau'$.

For $n\in I$, let $f_n\colon \R_+\to\R_+$ be defined by $f_n(t) = t 1_{[(1+\epsilon)^n,(1+\epsilon)^{n+1}[}(t)$ and put $h_n = c_n f_n(h)$, where $c_n = \tau(p_n)\tau'(p_n)^{-1}$. Then $h_n \in Z(M_n)_+$ is invertible in $M_n$ with spectrum $\sigma(h_n) \subseteq [c_n(1+\epsilon)^n,c_n(1+\epsilon)^{n+1}]$ and
$$
\tau'_n(x) = \tau_n(h_nx) \quad \text{for every } x\in M_n.
$$
By the first part of the proof applied to $M_n$ we get that $\Lambda_\WH(M_n,\tau_n') \leq \Lambda_\WH(M_n,\tau_n)(1+\epsilon)$ for every $n\in I$. Putting things together we obtain
$$
\Lambda_\WH(M,\tau') = \sup_{n\in I} \Lambda_\WH(M_n,\tau_n') \leq \sup_{n\in I} \Lambda_\WH(M_n,\tau_n)(1+\epsilon) = \Lambda_\WH(M,\tau) (1+\epsilon).
$$
This proves \eqref{eq:trace-proof}, and the proof is complete.

\end{proof}

\section{An example}\label{sec:application}

In this section we give an example of two von Neumann algebras, in fact $\II_1$ factors arising from discrete groups, with different weak Haagerup constants. None of the other approximation properties mentioned in the introduction (see Figure~\ref{fig:aps}) are useful as invariants to distinguish precisely these two factors (see Remark~\ref{rem:other-APs}).

It is well-known that if $\Gamma$ is an infinite discrete group, then $L(\Gamma)$ is a $\II_1$ factor if and only if all conjugacy classes in $\Gamma$ are infinite except for the conjugacy class of the neutral element. Such groups are called ICC (infinite conjugacy classes).

It is known from \cite{MR0147566} that every arithmetic subgroup of $\Sp(1,n)$ is a lattice. Let $\mb H_{\mr{int}}$ be the quaternion integers $\Z + \Z i + \Z j + \Z k$ inside the quaternion division ring $\mb H$, and let $n\geq 2$ be fixed. Then the group $\Gamma$ consisting of matrices in $\Sp(1,n)$ with entries in $\mb H_{\mr{int}}$ is an arithmetic subgroup of $\Sp(1,n)$ and hence a lattice. To be explicit, $\Gamma$ consists of $(n+1)\times (n+1)$ matrices with entries in $\mb H_{\mr{int}}$ that preserve the Hermitian form

$$
h(x,y) =  x_0 \overline{y_0} - \sum_{k=1}^n x_m \overline{y_m}, \quad x = (x_i),\ y=(y_i) \in \mb H^{n+1}.
$$
Here $\mb H^{n+1}$ is regarded as a right $\mb H$-module. If $I$ denotes the identity matrix in $\Sp(1,n)$, then the center of $\Sp(1,n)$ is $\{\pm I\}$, and it is proved in \cite[p.~547]{MR996553} that $\Gamma_0 = \Gamma / \{\pm I\}$ is an ICC group.

Let $H = \Z/2 \wr \F_2$ be the wreath product of $\Z/2$ and $\F_2$ (see Section~\ref{sec:introduction}). Then $H$ is ICC (see \cite[Corollary~4.2]{icc-extension}) and the direct product group $\Gamma_1 = \Gamma_0 \times H$ is also ICC (see \cite[p.~74]{icc-extension}).

Let $\Gamma_2 = \Z^2\rtimes\SL_2(\Z)$. It is well-known that $\Gamma_2$ is ICC and a lattice in $\R^2\rtimes\SL_2(\R)$. We claim that the $\II_1$ factors $L(\Gamma_1)$ and $L(\Gamma_2)$ are not isomorphic. Indeed, we show below that their weak Haagerup constants differ. Since both von Neumann algebras are $\II_1$ factors, there is a unique trace on each of them, so any isomorphism would necessarily be trace-preserving.

Using Theorem~\ref{thm:vna}, Proposition~\ref{prop:products}/Remark~\ref{rem:products}, Proposition~\ref{prop:modK},  Corollary~\ref{cor:lattice} and Theorem~\ref{thm:WA-Lie-rank-one} we get
\begin{align*}
\Lambda_\WH(L(\Gamma_1)) 
&= \Lambda_\WH(\Gamma_1) \\
&= \Lambda_\WH(\Gamma_0)\Lambda_\WH(H) \\
&= \Lambda_\WH(\Gamma) \\
&= \Lambda_\WH(\Sp(1,n)) \\
&\leq \Lambda_\WA(\Sp(1,n)) \\
&= 2n-1 < \infty.
\end{align*}
In \cite[Theorem~A]{HK-WH-examples} it is proved that $\R^2\rtimes\SL_2(\R)$ does not have the weak Haagerup property. Thus, using also Theorem~\ref{thm:vna} and Corollary~\ref{cor:lattice} we get
$$
\Lambda_\WH(L(\Z^2 \rtimes \SL_2(\Z))) = \Lambda_\WH(\Z^2 \rtimes \SL_2(\Z)) = \Lambda_\WH(\R^2 \rtimes \SL_2(\R)) = \infty.
$$
In view of Theorem~\ref{thm:hereditary2} this shows that $L(\Gamma_2)$ cannot be not embedded into any corner of any subalgebra of $L(\Gamma_1)$. In particular, $L(\Gamma_1)$ and $L(\Gamma_2)$ are not isomorphic.

\begin{rem}\label{rem:other-APs}
We remark that $\Gamma_1$ and $\Gamma_2$ do not have the Haagerup property. Also,
$$
\Lambda_\WA(L(\Gamma_1)) = \Lambda_\WA(L(\Gamma_2)) = \infty,
$$
and $\Gamma_1$ and $\Gamma_2$ both have AP. Thus, none of these three approximation properties distinguish $L(\Gamma_1)$ and $L(\Gamma_2)$.
\end{rem}

\appendix
\appendixpage
The appendices contain a collection of results that are used to show the equivalence of several definitions of the weak Haagerup property and of weak amenability. The results are certainly known to experts, but some of the results below do not appear explicitly or in this generality in the literature.

In all of the following $G$ is a locally compact group equipped with left Haar measure $dx$. For definitions concerning the Fourier algebra $A(G)$, the Herz-Schur multipliers $B_2(G)$ and its predual $Q(G)$ we refer to Section~\ref{sec:prelim}.

\section{Topologies on the unit ball of \texorpdfstring{$B_2(G)$}{B2(G)}}\label{app:topologies}
We are concerned with three different topologies on bounded sets in $B_2(G)$ besides the norm topology: The first topology is the weak$^*$-topology, where we view $B_2(G)$ as the dual space of $Q(G)$. It will be referred to as the $\sigma(B_2,Q)$-topology. The second topology is the locally uniform topology, i.e., the topology determined by uniform convergence on compact subsets of $G$. The third topology is the point-norm topology, where we think of elements in $B_2(G)$ as operators on $A(G)$. The following lemma reveals the relations between these topologies.

\begin{lem}\label{lem:A1}
Let $(u_\alpha)$ be a net in $B_2(G)$ and let $u\in B_2(G)$.
\begin{enumerate}
	\item If $\|(u_\alpha -u) w\|_A\to 0$ for every $w\in A(G)$, then $u_\alpha \to u$ uniformly on compacts.
  \item If the net is bounded and $u_\alpha\to u$ uniformly on compacts, then $u_\alpha\to u$ in the $\sigma(B_2,Q)$-topology.
\end{enumerate}
\end{lem}
\begin{proof}
Suppose $\|(u_\alpha -u) w\|_A\to 0$ for every $w\in A(G)$, and let $L\subseteq G$ be a compact subset. By \cite[Lemma~3.2]{MR0228628} there is a $w\in A(G)$ which takes the value 1 on $L$. Hence
$$
\sup_{x\in L} |u_\alpha(x) - u(x)| \leq \|(u_\alpha-u)w\|_\infty \leq \|(u_\alpha-u)w\|_A \to 0.
$$
This proves (1).

Suppose $u_\alpha\to u$ uniformly on compacts. Since the net $(u_\alpha)$ is bounded, and $C_c(G)$ is dense in $Q(G)$, it will suffice to prove $\la u_\alpha,f\ra \to \la u,f\ra$ for every $f\in C_c(G)$. Let $L = \supp f$. Then since $u_\alpha \to u$ uniformly on $L$, we obtain
$$
\la f,u_\alpha\ra = \int_L f(x)u_\alpha(x) \ dx  \to \int_L f(x)u(x) \ dx = \la f,u\ra.
$$
This proves (2).
\end{proof}

\begin{rem}
In the proof of (2), the assumption of boundedness is essential. In general, there always exist (possibly unbounded) nets $(u_\alpha)$ in $A(G) \subseteq B_2(G)$ converging to $1$ uniformly on compacts (use \cite[Lemma~3.2]{MR0228628}), but for groups without the Approximation Property (AP) such as $\SL_3(\Z)$ no such net can converge to 1 in the $\sigma(B_2,Q)$-topology (see \cite{MR1220905} and \cite[Theorem~C]{MR2838352}).
\end{rem}

\begin{lem}\label{lem:B2-unitball}
The unit ball of $B_2(G)$ is closed in $C(G)$ under locally uniform convergence and even pointwise convergence.
\end{lem}
\begin{proof}
This is obvious from the equivalence (1)$\iff$(2) in Proposition~\ref{prop:Gilbert}.
\end{proof}

\begin{lem}\label{lem:unit-ball-w-closed}
The unit ball of $B_2(G)$ is closed in $B_2(G)$ in the $\sigma(B_2,Q)$-topology.
\end{lem}
\begin{proof}
This is a consequence of Banach-Alaoglu's Theorem.
\end{proof}

\section{Average tricks}\label{app:average}

\subsection{The convolution trick}

In all of the following $h$ is a continuous, non-negative, compactly supported function on $G$ such that $\int h(x) \ dx = 1$. Such functions exist, and if $G$ is a Lie group, one can even take $h$ to be smooth.

The convolution trick consists of replacing a given convergent nets $(u_\alpha)$ in $B_2(G)$ with the convoluted nets $h * u_\alpha$ to obtain convergence in a stronger topology. Recall that the convolution of $h$ with $u\in L^p(G)$ is defined by
$$
(h*u)(x) = \int_G h(y)u(y^{-1}x) \ dy = \int_G h(xy)u(y^{-1}) \ dy, \quad x\in G.
$$

\begin{lem}[The convolution trick -- Part I]\label{lem:B1}
Let $u \in C(G)$ be given and let $h$ be as above.
\begin{enumerate}
\item If $u\in C_c(G)$, then $ h*u\in C_c(G)$.
\item If $u\in C_0(G)$, then $ h*u\in C_0(G)$.
\item If $u$ is uniformly bounded, then $\|h*u\|_\infty \leq \| u\|_\infty$.
\item If $u\in A(G)$, then $ h*u\in A(G)$ and $\| h*u\|_A \leq \| u\|_A$.
\item If $ u\in B_2(G)$, then $ h*u\in B_2(G)$ and $\| h*u\|_{B_2(G)} \leq \| u\|_{B_2(G)}$.
\item If $G$ is a Lie group and $h\in C^\infty_c(G)$, then $ h*u\in C^\infty(G)$.
\end{enumerate}
\end{lem}
\begin{proof}
\mbox{}

We leave (1)--(3) as an exercise.

(4) If $u\in A(G)$, then $u = f * \check g$ for some $f,g\in L^2(G)$ with $\|u\|_A = \|f\|_2\ \|g\|_2$. Then $h * u = (h * f) * \check g$. Since $h * f\in L^2(G)$ with $\|h * f\|_2 \leq \|f\|_2$ (see \cite[p. 52]{MR1397028}), it follows that $h * u\in A(G)$ with $\|h * u\|_A \leq \|u\|_A$.

(5) We use the characterization of Herz-Schur multipliers given in Proposition~\ref{prop:Gilbert}. Given $y\in G$ we let $y.u$ be defined by $(y.u)(x) = u(y^{-1}x)$ for $x\in G$. Clearly, $y.u\in B_2(G)$ and $\|y.u\|_{B_2} = \|u\|_{B_2}$.

Let $n\in\N$ and $x_1,\ldots,x_n\in G$ in be given and let $m\in M_n(\C)$ be the $n\times n$ matrix
$$
m = u(x_j^{-1}x_i)_{i,j=1}^n.
$$
More generally, for any $y\in G$, let $y.m$ denote the matrix
$$
y.m = u(y^{-1} x_j^{-1}x_i)_{i,j=1}^n
$$
Clearly, $\|y.m\|_S \leq \|y.u\|_{B_2} = \|u\|_{B_2}$ and $y\mapsto y.m$ is continuous from $G$ into $M_n(\C)$, when $M_n(\C)$ is equipped with the Schur norm. Thus, by usual Banach space integration theory,
$$
(h * u)(x_j^{-1}x_i)_{i,j=1}^n = \int_G h(y) (y.m) \ dy
$$
has Schur norm at most $\|u\|_{B_2}$. By Proposition~\ref{prop:Gilbert} (2) it follows that the Herz-Schur norm of $h*u$ satisfies
$$
\| h*u\|_{B_2(G)} \leq \| u\|_{B_2(G)}.
$$

(6) This is elementary.
\end{proof}

The proof of (1) in the lemma below is taken from \cite[p. 510]{MR996553}. Although they in \cite{MR996553} assume that $u_\alpha \in A(G)$ and $u = 1$, the proof carries over without changes.

\begin{lem}[The convolution trick -- Part II]\label{lem:B2}
Let $(u_\alpha)$ be a bounded net in $B_2(G)$, let $u\in B_2(G)$ and let $h$ be as above. We set
$$
v_\alpha = h * u_\alpha \quad\text{and}\quad v = h * u.
$$
\begin{enumerate}
\item If $ u_\alpha \to  u$ uniformly on compacts then $\| (v_\alpha -v )w \|_A\to 0$ for every $ w\in A(G)$.
\item If $ u_\alpha \to  u$ in the $\sigma(B_2,Q)$-topology then $ v_\alpha \to  v$ uniformly on compacts.
\end{enumerate}
\end{lem}
\begin{proof}
\mbox{}\newline
(1) Assume $u_\alpha \to  u$ uniformly on compacts. Since the net $(u_\alpha)$ is bounded in $B_2$-norm, and since $A(G) \cap C_c(G)$ is dense in $A(G)$, it follows from \eqref{eq:multiplier-norm} that it will suffice to prove that
$$
\| (v_\alpha -v )w \|_A\to 0
$$
for every $w\in A(G) \cap C_c(G)$. We let $S$ denote the compact set $\supp(h)^{-1}\supp(w)$ and $1_S$ its characteristic function. Then if $x\in \supp(w)$
$$
(h * u_\alpha)(x) = \int_G h(y) u_\alpha(y^{-1}x) \ dy = \int_G h(y) (1_Su_\alpha)(y^{-1}x)\ dy
$$
because if $y^{-1}x\notin S$, then $h(y) = 0$. It follows that
\begin{align}\label{eq:vw1}
(v_\alpha w)(x) = ((h*1_S u_\alpha) w)(x).
\end{align}
Note that \eqref{eq:vw1} actually holds for all $x\in G$, since if $x\notin\supp(w)$, then both sides vanish. Similarly one can show
$$
(vw)(x) = ((h*1_S u) w)(x) \quad\text{ for all } x\in G.
$$
By assumption, $1_S u_\alpha \to 1_S u$ uniformly, and hence
$$
\| h * 1_S u_\alpha - h* 1_S u \|_A \leq \|h\|_2 \ \|\widecheck{1_S u_\alpha} - \widecheck{1_S u} \|_2 \to 0.
$$
Since multiplication in $A(G)$ is continuous we also have
$$
\| (v_\alpha - v)w \|_A\to 0.
$$
This completes the proof of (1).

(2) For each $x\in G$, let $t(x) = h_x \in C_c(G)$ be the function $h_x(y) = h(xy)$. The map $t: G\to C_c(G)$ is continuous, when $C_c(G)$ is equipped with the $L^1$-norm (see \cite[Proposition~2.41]{MR1397028}). Since the $Q$-norm is dominated by the $L^1$-norm, it follows that $t$ is continuous into $Q(G)$.

Assume that $u_\alpha \to  u$ in the $\sigma(B_2,Q)$-topology, and let $L\subseteq G$ be compact. Since the net $(u_\alpha)$ is bounded, the convergence is uniform on compact subsets of $Q(G)$. By continuity of $t$, the set
$$
T = \{ h_x\in C_c(G) \mid x\in L\}
$$
is a compact subset of $Q(G)$. Hence
$$
(h*u_\alpha)(x) = \la h_x, \check u_\alpha\ra \to \la h_x, \check u\ra = (h*u)(x)
$$
uniformly on $L$.
\end{proof}

\begin{rem}\label{rem:conv-trick}
In applications, $u$ will often be the constant function $1\in B_2(G)$, and in that case $h * u = 1$.
\end{rem}

\begin{lem}
Let $(u_\alpha)$ be a net in $B_2(G)$ and let $u\in B_2(G)$. We set
$$
v_\alpha = h * u_\alpha \quad\text{and}\quad v = h * u.
$$
\begin{enumerate}
\item If $ u_\alpha \to  u$ uniformly on compacts then $ v_\alpha \to  v$ uniformly on compacts.
\item If $ u_\alpha \to  u$ in the $\sigma(B_2,Q)$-topology then $v_\alpha \to  v$ in the $\sigma(B_2,Q)$-topology.
\end{enumerate}
\end{lem}
\begin{proof}
\mbox{}

(1) For any subset $L\subseteq G$ we observe that 
\begin{align*}
\sup_{x\in L} | v_\alpha(x) - v(x)|
\leq \sup_{x\in \supp(h)^{-1}L} | u_\alpha(x) - u(x)| \to 0.
\end{align*}
If $L$ is compact, then $\supp(u)^{-1}L$ is compact as well. This is sufficient to conclude (1).

(2) Let $\Delta:G\to\R_+$ be the modular function. When $f\in L^1(G)$ we let
$$
(Rf)(y) = \Delta(y^{-1}) \int_G f(x) h(xy^{-1}) \ dx.
$$
It is not hard to show that $\|Rf\|_1 \leq \|f\|_1$ and in particular $Rf\in L^1(G)$. We observe that if $w\in B_2(G)$ then
\begin{align*}
\la f , h*w \ra
&= \int_G f(x) (h*w)(x) \ dx \\
&= \int_{G \times G} f(x) h(xy^{-1}) w(y) \Delta(y^{-1})\ dy dx \\
&= \la Rf,w\ra.
\end{align*}
It now follows from Lemma~\ref{lem:B1} (5) that $R$ extends uniquely to a linear contraction $R:Q(G) \to Q(G)$, and that the dual operator $R^*:B_2(G)\to B_2(G)$ satisfies $R^* w = h*w$. Since $R^*$ is weak$^*$-continuous we conclude
$$
\la f,v_\alpha\ra = \la f, R^*u_\alpha\ra \to \la f, R^*u\ra = \la f,v_\alpha\ra
$$
for any $f\in Q(G)$ as desired.
\end{proof}

\subsection{The bi-invariance trick}

In all of the following $K$ is a compact subgroup of $G$ equipped with normalized Haar measure $dk$.

\begin{lem}[The bi-invariance trick -- Part I]\label{lem:B3}
Let $u\in C(G)$ or $u\in L^1(G)$ be given, and define
\begin{align}\label{eq:sharp}
 u^K(x) = \int_{K\times K}  u(kxk') \ dkdk', \qquad x\in G.
\end{align}
Then $u^K$ is a $K$-bi-invariant function on $G$. Moreover, the following holds.

\begin{enumerate}
\item If $ u\in C(G)$, then $ u^K\in C(G)$.
\item If $ u\in C_c(G)$, then $ u^K\in C_c(G)$.
\item If $ u\in C_0(G)$, then $ u^K\in C_0(G)$.
\item $\| u^K\|_\infty \leq \| u\|_\infty$.
\item If $ u\in L^1(G)$, then $ u^K\in L^1(G)$ and $\|u^K\|_1 \leq \|u\|_1$.
\item If $ u\in A(G)$, then $ u^K\in A(G)$ and $\| u^K\|_A \leq \| u\|_A$.
\item If $ u\in B_2(G)$, then $ u^K\in B_2(G)$ and $\| u^K\|_{B_2(G)} \leq \| u\|_{B_2(G)}$.
\item If $G$ is a Lie group and $ u\in C^\infty(G)$, then $ u^K\in C^\infty(G)$.
\end{enumerate}
\end{lem}
\begin{proof}
\mbox{}

(1) Suppose $u\in C(G)$. To simplify matters, we first show that $u_K$ given by
$$
u_K(x) = \int_K  u(kx) \ dk, \qquad x\in G
$$
is a continuous function on $G$. A similar argument will then show that $u^K$ is continuous, because
$$
u^K(x) = \int_K  u_K(xk) \ dk, \qquad x\in G.
$$
Let $x\in G$ and $\epsilon > 0$ be given. We will find a neighborhood $V$ of the identity such that
$$
|u_K(x) - u_K(zx)| \leq \epsilon \quad\text{for all } z\in V.
$$
Actually, if will be sufficient to verify that
$$
|u(kx) - u(kzx)| \leq \epsilon \quad\text{for all } z\in V \text{ and } k\in K.
$$
For each $k\in K$, the function $x\mapsto u(kx)$ is continuous, so there exists a neighborhood $U_k$ of the identity such that
$$
|u(kx) - u(kzx)| \leq \epsilon/2 \quad\text{for all } z\in U_k.
$$
Let $V_k$ be a neighborhood of the identity such that $V_k V_k \subseteq U_k$. Observe that the sets $k V_k$ where $k\in K$ together cover $K$, so by compactness
$$
K \subseteq k_1 V_{k_1} \cup \cdots \cup k_n V_{k_n}
$$
for some $k_1,\ldots,k_n \in K$. Let $V = \bigcap_{i=1}^n V_{k_i}$. Now, let $k\in K$ and $z\in V$ be arbitrary, and choose $i\in \{1,\ldots,n\}$ such that $k\in k_i V_{k_i}$. Note that then $k_i^{-1} k \in V_{k_i} \subseteq U_{k_i}$ and $k_i^{-1}k z \in V_{k_i} V_{k_i} \subseteq U_{k_i}$. Thus,
$$
|u(kx) - u(kzx)| \leq |u(k_i(k_i^{-1}k)x) - u(k_ix)| + |u(k_ix) - u(k_i(k_i^{-1}kz)x)|         \leq \epsilon 
$$
as desired.

(2)-(4) we leave as an exercise.

(5) Recall (see \cite[Section~2.4]{MR1397028}) the fundamental relation of the modular function~$\Delta$,
$$
\Delta(y) \int_G f(xy) \ dx = \int_G f(x)\ dx.
$$
Also, $\Delta|_K = 1$, since $K$ is compact. We now compute
\begin{align*}
\int_G |u^K(x)| \ dx
& \leq \int_{K\times K} \int_G |u(kxk')| \ dxdkdk' \\
& = \int_{K\times K} \Delta(k') \int_G |u(x)| \ dxdkdk' \\
& = \|f\|_1.
\end{align*}
This proves (5).

(6) It suffices to note that $A(G)$ is a Banach space, that left and right translation on $A(G)$ is continuous and isometric, and then apply usual Banach space integration theory.

(7) This is mentioned in \cite{MR996553}. An argument similar the proof of Lemma~\ref{lem:B1} (5) applies. Alternatively, one can use the proof from \cite[Section~3]{steenstrup1}.

(8) This is elementary.
\end{proof}

\begin{lem}[The bi-invariance trick -- Part II]\label{lem:B4}
Let $(u_\alpha)$ be a net in $B_2(G)$ and let $u\in B_2(G)$. We set
$$
u^K_\alpha(x) = \int_{K\times K}  u_\alpha(kxk') \ dkdk' \quad\text{and}\quad u^K(x) = \int_{K\times K}  u(kxk') \ dkdk'
$$
\begin{enumerate}
\item If $u_\alpha \to  u$ uniformly on compacts then $ u^K_\alpha \to  u^K$ uniformly on compacts.
\item If $u_\alpha \to  u$ in the $\sigma(B_2,Q)$-topology then $ u^K_\alpha \to  u^K$ in the $\sigma(B_2,Q)$-topology.
\end{enumerate}
\end{lem}
\begin{proof}
\mbox{}

(1) Suppose $u_\alpha \to u$ uniformly on compacts. Let $L\subseteq G$ be compact. Then since $u_\alpha \to u$ uniformly on the compact set $KLK$, we have
\begin{align*}
\sup_{x\in L} | u^K_\alpha(x) - u^K(x) |
& \leq \sup_{x\in L} \int_{K\times K} |u_\alpha(kxk') - u(kxk')| \ dkdk' \\
& \leq \sup_{y\in KLK} | u_\alpha(y) - u(y) | \to 0.
\end{align*}
This shows that $u^K_\alpha \to u^K$ uniformly on $L$.

(2) This is proved in \cite[Lemma~2.5]{MR3047470}. We sketch the proof here. Observe that 
$$
\la f,v^K \ra = \la f^K, v\ra
$$
for any $v\in B_2(G)$ and $f\in L^1(G)$. Thus $\|f^K\|_Q \leq \|f\|_Q$ by Lemma~\ref{lem:B3} (7), and the map $f\mapsto f^K$ extends uniquely to a linear contraction $R:Q(G) \to Q(G)$. The dual operator $R^* : B_2(G)\to B_2(G)$ obviously satisfies $R^* v = v^K$ and is weak$^*$-continuous. Hence
$$
\la f , u_\alpha^K \ra = \la f,R^*u_\alpha\ra \to \la f,R^* u\ra = \la f, u^K\ra
$$
for any $f\in Q(G)$ as desired.
\end{proof}

\section{Continuity of Herz-Schur multipliers}

U.~Haagerup has allowed us to include the following lemma whose proof is taken from Appendix A in the unpublished manuscript \cite{UH-preprint}.

\begin{lem}[\cite{UH-preprint}]\label{lem:HS-continuity}
Let $G$ be a locally compact group, let $ u:G\to\C$ be a function, and suppose there exist a separable Hilbert space $\mc H$ and two bounded Borel maps $P,Q:G\to\mc H$ such that
	$$
	 u(y^{-1}x) = \la P(x),Q(y) \ra \qquad\text{for all } x,y\in G.
	$$
Then $u$ is continuous, $u\in B_2(G)$ and
	$$
	\| u\|_{B_2} \leq \|P\|_\infty \|Q\|_\infty.
	$$
\end{lem}
\begin{proof}
We construct another Hilbert space $\mc K$ and two \emph{continuous} bounded maps $\hat P,\hat Q:G\to\mc K$ such that
	$$
	 u(y^{-1}x) = \la \hat P(x),\hat Q(y) \ra \qquad\text{for all } x,y\in G
	$$
	and
	$$
	\|\hat P\|_\infty \|\hat Q\|_\infty \leq \|P\|_\infty\|Q\|_\infty.
	$$
This will complete the proof in the light of Proposition~\ref{prop:Gilbert} (4).

Take $h\in C_c(G)$ satisfying $\|h\|_2 = 1$, and define
$$
\hat P(x), \hat Q(x) \in L^2(G,\mc H) \quad\text{for all } x\in G
$$
by
$$
\hat P(x)(z) = h(z)P(zx), \quad z\in G;
$$
$$
\hat Q(x)(z) = h(z)Q(zx), \quad z\in G.
$$
We find
\begin{align*}
\la \hat P(x),\hat Q(y)\ra &= \int_G |h(z)|^2 \la P(zx),Q(zy)\ra \ dz    \\
&= \int_G |h(z)|^2  u(y^{-1}x) \ dz  \\
&=  u(y^{-1}x).
\end{align*}
It is also easy to see that
$$
\sup_{x\in G} \| \hat P(x)\|_2 \leq \sup_{x\in G} \|P(x)\|, \quad
\sup_{x\in G} \| \hat Q(x)\|_2 \leq \sup_{x\in G} \|Q(x)\|,
$$
so in particular the maps $\hat P, \hat Q : G\to L^2(G,\mc H)$ are bounded. It remains only to check continuity of $x\mapsto \hat P(x)$ and $x\mapsto \hat Q(x)$. Let $\rho:G\to B(L^2(G))$ be the right regular representation
$$
\rho_x(f)(z) = \Delta^{1/2}(x) f(zx), \quad f\in L^2(G), \ x,z\in G.
$$
We let $R$ be the representation $\rho\otimes 1$ of $G$ on $L^2(G,\mc H) = L^2(G)\otimes \mc H$, that is,
$$
R_x(f)(z) = \Delta^{1/2}(x) f(zx), \quad f\in L^2(G,\mc H), \ x,z\in G.
$$
It is well-known that $\rho$ is strongly continuous (see Proposition 2.41 in \cite{MR1397028}), and hence $R$ is strongly continuous.

Suppose $x_n\to x$ in $G$. Then $R_{x_nx^{-1}} \hat P(x) \to \hat P(x)$ in $L^2(G,\mc H)$. Also,
\begin{align*}
\| R_{x_nx^{-1}} \hat P(x) - \hat P(x_n) \|^2
&= \int_G \| R_{x_nx^{-1}} \hat P(x)(z) - \hat P(x_n)(z) \|^2 \ dz  \\
&= \int_G \| \Delta(x_nx^{-1})^{1/2} h(zx_nx^{-1}) P(zx_n) - h(z)P(zx_n) \|^2 \ dz  \\
&= \int_G |\Delta(x_nx^{-1})^{1/2} h(zx_nx^{-1}) - h(z)|^2 \|P(zx_n) \|^2 \ dz  \\
&\leq \int_G |\Delta(x_nx^{-1})^{1/2} h(zx_nx^{-1}) - h(z)|^2 \|P\|_\infty^2 \ dz  \\
&= \|P\|_\infty^2 \| \rho_{x_nx^{-1}} h - h \|^2 \to 0.
\end{align*}
Hence $\hat P(x_n) \to \hat P(x)$ as desired. Continuity of $\hat Q$ is verified similarly.
\end{proof}

\section*{Acknowledgments}
The author wishes to thank Uffe Haagerup for numerous discussions on the subject and for allowing us to include his proof of Lemma~\ref{lem:HS-continuity}.

%\bibliographystyle{plain}
%\bibliography{knudbybib}

\end{document}